\documentclass[12pt]{amsart}

\usepackage{amsmath}
\usepackage{amsfonts}
\usepackage{amssymb} 

\usepackage{url} 
\usepackage{bbm}
\usepackage{xcolor}

\usepackage{multicol}
\usepackage{pst-fractal}
\usepackage{comment} 
\usepackage{hyperref,cleveref}
\usepackage[backend=biber, style=alphabetic, maxbibnames=99]{biblatex}
\usepackage[T1]{fontenc}
\usepackage{newpxtext,newpxmath} 
\usepackage[utf8]{inputenc}
		   \usepackage{palatino}

\usepackage[normalem]{ulem} 
\usepackage{tikz-cd}
\usepackage{enumerate}
\usepackage{layout}

\hypersetup{colorlinks=true,
linkcolor=blue, citecolor=blue, urlcolor=blue}
\usepackage{fullpage} 
\usepackage[margin=1in]{geometry}

\newtheorem{thm}{Theorem}[section]       

\newtheorem{cor}[thm]{Corollary}
\newtheorem{lem}[thm]{Lemma}
\newtheorem{prop}[thm]{Proposition}

\theoremstyle{definition}
\newtheorem{defn}[thm]{Definition}

\newtheorem{rmk}[thm]{Remark}

\numberwithin{equation}{section}

\DeclareMathOperator{\id}{Id}
\DeclareMathOperator{\Span}{span}

\renewcommand{\max}{max}
\DeclareMathOperator{\MAX}{MAX}
\DeclareMathOperator{\MIN}{MIN}

\DeclareMathOperator{\ball}{\text{Ball}}
\DeclareMathOperator{\sphere}{\text{Sphere}}

\newcommand{\ran}{\mathrm{ran} \,}

\newcommand{\N}{\mathbb{N}}
\newcommand{\R}{\mathbb{R}}
\newcommand{\C}{\mathbb{C}}

\newcommand{\e}{\epsilon}
\newcommand{\vr}{\varepsilon}

\renewcommand{\a}{\alpha}

\newcommand{\vp}{\varphi}

\newcommand{\expe}{\mathbb{E}}

\newcommand{\tr}{\text{tr}}

\makeatletter
\DeclareRobustCommand\smileotimes{\mathbin{\mathpalette\smile@otimes\relax}}
\newcommand{\smile@otimes}[2]{%
  \vbox{
    \ialign{##\cr
      \hidewidth$\m@th#1{}_\smile$\kern-\scriptspace\hidewidth\cr
      \noalign{\nointerlineskip\kern-1pt}
      $\m@th#1\otimes$\cr
    }%
  }%
}
\makeatother

\makeatletter
\DeclareRobustCommand\frownotimes{\mathbin{\mathpalette\frown@otimes\relax}}
\newcommand{\frown@otimes}[2]{%
  \vbox{
    \ialign{##\cr
      \hidewidth$\m@th#1{}_\frown$\kern-\scriptspace\hidewidth\cr
      \noalign{\nointerlineskip\kern-1pt}
      $\m@th#1\otimes$\cr
    }%
  }%
}
\makeatother

\newcommand{\tri}[1]{| \! | \! |{#1}| \! | \! |}

\newcommand{\norm}[1]{\left\lVert #1 \right\rVert}

\newcommand{\bb}[1]{\mathbb{#1}}
 
\newcommand{\cc}[1]{\mathcal{#1}}

\newcommand{\cbnorm}[1]{\left\lVert #1 \right\rVert_{\tx{cb}}}

\renewcommand{\sp}[2]{\langle #1, #2 \rangle} 
\renewcommand{\mp}[2]{\langle \langle #1, #2 \rangle \rangle} 

\newcommand{\ncconvexhull}[1]{\tx{nco}(#1)}
\newcommand{\clncconvexhull}[1]{\overline{\tx{nco}}(#1)}

\newcommand{\tx}[1]{\text{#1}}

\newcommand{\ket}[1]{\vert #1 \rangle}
\newcommand{\bra}[1]{\langle #1 \vert}
\newcommand{\op}[2]{\ket{#1}{\bra{#2}}}

\newcommand{\tra}[1]{\mathrm{tr}{#1}}

\newcommand{\oper}[1]{\mathbf{op}({#1})}

\def\query#1{\setlength\marginparwidth{55pt}%
\marginpar{\raggedright\fontsize{10}{10}\selectfont\itshape
\hrule\smallskip
{\textcolor{red}{#1}}\par\smallskip\hrule}}

\newcommand{\timur}[1]{{\textcolor{blue}{{\bf T.O:} #1}}}

\title{On Matricial Order Operator Spaces} 
\author{Roy Araiza and Timur Oikhberg}

\begin{document}
 
\begin{abstract}
We investigate the category of ``matricial order operator spaces,'' which generalize operator systems, being equipped with both matricial norms and matricial order.
For these objects, we develop duality theory. Taking a cue from the theory of ordered normed spaces, we introduce two important properties describing the interplay between order and norm -- ``normality'' and ``generation,'' and show that they are dual to each other. As examples, we consider operator systems (in particular, C$^*$-algebras), and Schatten spaces.
We also describe the minimal and maximal matricial order structures (which, again, turn out to be in duality), and show how Banach lattices can be equipped with such structures.
\end{abstract}

 \thanks{During this project the first author was supported as a J.L. Doob Research Assistant Professor at the University of Illinois at Urbana-Champaign, and as a Trond Mohn Stiftelse Guest Professor at Universitetet I Oslo. The authors wish to express their gratitude to E.~Bilokopytov and E.~Pernecka for many helpful comments.}
 \address{Department of Mathematics, University of Illinois at Urbana-Champaign, Urbana, IL}
 \email{raraiza@illinois.edu}
 \email{oikhberg@illinois.edu}
 
\maketitle

\tableofcontents

\section{Introduction}\label{intro}

The present paper is devoted to investigating matricial order operator spaces (defined below). Our goal is to unify two streams of Functional Analysis: (i) the theory of ordered normed spaces, and (ii) operator space theory.

Ordered spaces (such as function spaces) have been studied since the dawn of functional analysis. For a general approach, we refer the reader to e.g. \cite{batty1984positive} or \cite{Aliprantis-Tourky}. Below, we recall the facts and definitions relevant to this work.

Operator spaces arose in the late 1980s, as subspaces of $C^*$-algebras. Alternatively, one equips a given normed space with a matricial structure, satisfying the so called Ruan's Axioms. Accessible monographs on the topic include \cite{paulsen_book}, \cite{Pisier-Op-Spaces}, and \cite{effros-ruan-book}; again, the most pertinent results are recalled below.

From the beginning of operator space theory, people were interested in (\emph{unital}) \emph{operator systems} -- that is, involution-invariant subspaces of $B(H)$ which contain the identity, and hence naturally carry ``matricial order'' (see e.g.~\cite{paulsen_book}).
Instead of concrete embeddings into $B(H)$, one can give ``abstract'' descriptions of such objects, involving interactions of the unit, the matricial norms, and the matricial order, see \cite{choi1977injectivity}.
 
Recently, general (not necessarily unital) ordered operator spaces have come into focus, due to their connections to non-commutative convexity (see e.g. \cite{kennedy-kim-manor2025}).
Here, the order is required to satisfy certain conditions (which we make explicit below), but doesn't need to arise from a specific embedding into $B(H)$.
In fact, \cite{Karn} already notes the difficulty of obtaining such embeddings of duals of $C^*$-algebras.

All this points to the necessity of developing a general theory of ``matricial order'' on operator spaces, which is the purpose of the present work.
We begin this paper by laying the necessary foundation. In \Cref{s:order Banach}, we briefly introduce the notion of an ordered Banach space, and review ``order'' versions of local reflexivity.
Apart from being of independent interest, local reflexivity will also be used to describe the duality between minimal and maximal matricial order structures in \Cref{min max structures}. 
\Cref{prelim} examines matricially normed spaces (operator spaces), and also matricial convexity and duality.

We then move on to the core of the present work -- namely, the examination of matricial order. \Cref{s: basic properties} defines \emph{matricial order operator spaces}, and introduces their basic properties, such as normality and generation (their classical counterparts are described in \Cref{s:order Banach}).
In \Cref{positivisation} we describe a way to renorm a normal, geneating matricial order space to make it $1$-normal and $1$-generating.

It is easy to observe that the dual of a matrical order operator space can be equipped with a matricial order operator space structure. In \Cref{ss:duality} we show that the dual of $\cc X$ is normal if and only if $\cc X$ itself is generating, and vice versa.

Taking a cue from the theory of operator spaces, in \Cref{min max structures} we define the maximal and minimal matricial order structures on a given ordered Banach space, and show that they are in duality. We also develop criteria for normality and generation of these ``maximal'' and ``minimal'' matricial order operator spaces.

\Cref{ss:examples} is devoted to investigating ``natural'' examples of matricial order operator spaces -- namely, $C^*$-algebras (and operator systems more generally) and Schatten spaces. In \Cref{lattice} we study a new class of matricial order operator spaces obtained by equipping Banach lattices with their minimal or the maximal structures.

For functional analysis background, see \cite{AK}. We mostly adhere to the standard notation; however, some special symbols are used, to wit:

$\bullet$ If $\cc X$ is a vector space, we shall use $\cc X^\flat$ for a space of linear functionals on it. 
The dual action (linear in both variables) is denoted by $\langle \cdot, \cdot \rangle : \cc X^\flat \times \cc X \to \C$.

If $\cc X$ is a topological vector space, then $\cc X^\flat$ stands for \emph{the} space of \emph{continuous} functionals. If $\cc X$ is a normed space, then we equip $\cc X^\flat$ with the usual dual norm.

We adopt this unusual notation, to avoid giving double duty to ``$*$:'' when $x$ is an operator on a Hilbert space (or a matrix), $x^*$ shall stand for the adjoint of $x$.

$\bullet$ Following the conventions of mathematical physics, we use the ``bra'' and ''ket'' for Hilbert space inner products: the map $H \times H \to \C : (\xi, \eta) \mapsto \langle \xi | \eta \rangle$ is linear in the first variable, and antilinear in the second one.
For Hilbert spaces $H$ and $K$, $T \in B(H,K)$, and $\eta \in H$, we write $T | \eta \rangle$ for $T\eta$.
Further, we shall sometimes denote the canonical basis for $\ell_2^n$ by $(|i\rangle)_{i=1}^n$.

$\bullet$ For $1 \leq p \leq \infty$, $S_p^n$ will denote the space of $n \times n$ matrices, with the Schatten $p$-norm (operator norm for $p=\infty$), and order determined by positive definiteness of a matrix. More generally, for $m \times n$ matrices we use $S_p^{m,n}$.
The Schatten space $S_p$ is viewed as the completion of $\cup_n S_p^n$.

$\bullet$ For a normed space $(\cc X, \| \cdot \|)$, $\ball(\cc X)$ stands for the closed unit ball $\{x \in \cc X : \|x\| \leq 1\}$, and $\sphere(\cc X)$ -- for the unit sphere $\{x \in \cc X : \|x\|=1\}$.
 
Further definitions will be introduced in the text as necessary.


\section{Ordered vector spaces and local reflexivity}\label{s:order Banach}

In this section, we discuss ordered real Banach spaces, and, in particular, the specifics of local reflexivity therein. For unexplained facts and notation, we refer the reader to \cite{Aliprantis-Tourky} and \cite{batty1984positive}.

\begin{defn}\label{def:wedge and cone}
Suppose $\cc X$ is a \emph{real} topological vector space. Following \cite{Aliprantis-Tourky}, we say that a closed set $\cc X^+ \subset \cc X$ is a \emph{wedge} in $\cc X$ if (i) $\cc X^+ + \cc X^+ \subset \cc X^+$, (ii) $\R^+ \cc X^+ = \cc X^+$.
$\cc X^+$ is called a (\emph{positive}) \emph{cone} if, in addition, it is \emph{pointed} -- that is, $\cc X^+ \cap (-\cc X^+) = \{0\}$.
\end{defn}

Such a wedge turns $\cc X$ into an \emph{ordered space}: we say that $x \geq y$ if $x-y \in \cc X^+$. If $\cc X^+$ is a cone, then $x=y$ iff $x \geq y$ and $y \geq x$.

For $A \subset \cc X$, we use a shorthand $A^+$ for $A \cap \cc X^+$.

The dual $\cc X^\flat$ of an ordered space $\cc X$ is turned into an ordered space by defining $\cc X^{\flat+}$ to be the set of all $x^\flat \in \cc X^\flat$ for which $\langle x^\flat,x \rangle \geq 0$ for any $x \geq 0$. Note that $\cc X^{\flat+}$ is $\sigma(\cc X^\flat, \cc X)$-closed (hence norm closed if $\cc X$ is a normed space).
Clearly $\cc X^{\flat+}$ is a wedge; by \cite[Theorem 2.13]{Aliprantis-Tourky}, $\cc X^{\flat+}$ is a cone iff $\cc X^+ - \cc X^+$ is $\sigma(\cc X, \cc X^\flat)$-dense in $\cc X$.

By \cite[Theorem 2.13]{Aliprantis-Tourky}, $x \in \cc X^+$ iff $\langle x^\flat,x \rangle \geq 0$ for any $x^\flat \in \cc X^{\flat+}$. Consequently, an element $x \in \cc X$ belongs to $\cc X^+$ if and only if it belongs to $\cc X^{\flat\flat+}$ (here and throughout the paper, $\cc X$ is viewed as canonically embedded into $\cc X^{\flat\flat}$).

Banach-Alaoglu Theorem states that a Banach space $\cc X$ is weak$^*$-dense in $\cc X^{\flat\flat}$. Likewise, \cite[Theorem 2.13]{Aliprantis-Tourky} gives us:

\begin{prop}\label{p:Banach Alaoglu}
 If $\cc X$ is an ordered Banach space, then $\cc X^+$ is weak$^*$ dense in $\cc X^{\flat\flat+}$. 
\end{prop}

Later (\Cref{c:goldstine from behrends}) we shall see that an analogue of Goldstine Theorem holds as well.

\begin{defn}\label{def:normal generating}
An ordered space $\cc X$ is said to be:
\begin{itemize}
\item $C$-\emph{normal} if, for any $x \in \cc X$ and $y \in \cc X^+$, $-y \leq x \leq y$ implies $C\|y\| \geq \|x\|$.
\item $C$-\emph{generating} if, for any $x \in \cc X$ and $\vr > 0$, there exists $y \in \cc X^+$ so that $-y \leq x \leq y$, and $\|y\| \leq C\|x\| + \vr$.
\end{itemize}
We say that a space is \emph{normal} (\emph{generating}) if it is $C$-normal (resp.~$C$-generating).
\end{defn}

In \cite[Section 1.3]{batty1984positive}, the terms ``$C$-absolutely monotone'' and ``approximately $C$-absolutely dominating'' were used instead of ``$C$-normal'' and ``$C$-generating,'' respectively.

Note that the normality of $\cc X$ implies that $\cc X^+$ is a cone. Indeed, if $x \in \cc X^+ \cap (- \cc X^+)$, then $x \leq 0$; if $\cc X$ is $C$-normal, then $\|x\| \leq C\|0\| = 0$, so $x = 0$.
Likewise, the characterization of $\cc X^{\flat+}$ being a cone given above shows that, if $\cc X$ is generating, then $\cc X^{\flat+}$ is a cone.

A straightforward compactness argument shows that a finite dimensional ordered space $E$ is normal iff $E^+$ is a cone, and is generating iff $E^+$ has non-empty interior.

Re-working examples from \Cref{s: basic properties} below, one see that $(S_p^m)_h$ (the Hermitian, or self-adjoint, part of the $m \times m$ Schatten space $S_p^m$) is $1$-normal and $1$-generating. Any Banach lattice is $1$-normal and $1$-generating as well. By \cite{vanGaans2004embeddings}, for an ordered Banach space $\cc X$, the following are equivalent: (1) $\cc X$ is $1$-normal, and (2) there exists a Banach lattice $\cc Y$ and a bipositive isometry $j : \cc X \to \cc Y$ (a map $T$ is called \emph{bipositive} if $Tx \geq 0$ iff $x \geq 0$).

For future use, we examine ``order-consistent'' local reflexivity (based on \cite{Behrends} and \cite{CG}). Throughout, for a normed space $\cc Z$, $C \subset \cc Z$, and $\vr > 0$, we write $C_\vr = \big\{ z \in \cc Z : {\mathrm{dist}}(z,C) \leq \vr \big\}$. To set up the problem, suppose:
\begin{enumerate}
 \item $E$, $\cc X$, $\cc Y_i$ ($1 \leq i \leq m$), and $\cc Z_j$ ($1 \leq j \leq n$) are Banach spaces, with $E$ finite dimensional.
 \item $A_i : B(E,\cc X) \to \cc Y_i$ are bounded operators; $y_i \in \cc Y_i$.
 \item $B_j : B(E,\cc X) \to \cc Z_j$ are bounded operators; $C_j \subset \cc Z_j$ are closed convex sets.
\end{enumerate}

\begin{thm}\label{t:useless-beh}
 Keep the above notation, and pick $T \in B(E, \cc X^{\flat\flat})$. Then the following are equivalent:
 \begin{enumerate}
 \item For and any $\vr > 0$, there exists $S \subset B(E,\cc X)$ so that $(1-\vr) \|T\| \leq \|S\| \leq (1+\vr) \|T\|$, 
 $A_i S = y_i$, and $B_j S \in (C_j)_\vr$ for any $i$ and $j$.
  \item $T$ is weak$^*$ continuous on the weak$^*$ closure of $\sum_i \ran A_i^\flat$, $A_i^{\flat\flat} T = y_i$ for any $i$, and $B_j^{\flat\flat} T \in \overline{C_j}^{{\textrm{weak}^*}}$ for any $j$.
 \end{enumerate}
\end{thm}

Here, the statement ``$T$ is weak$^*$ continuous on $\cdots$'' is understood in the sense that $\ran A_i^\flat$ lies in $B(E, \cc X)^\flat = E \widehat{\otimes} \cc X^\flat$, and $T$ is naturally identified with an element of $B(E, \cc X)^{\flat\flat}$.

\begin{proof}[Sketch of a proof]
 Theorem 2.3 of \cite{Behrends} (itself deduced from \cite[Theorem 2.1]{Behrends}) provides the proof when $T$ is an isometric embedding of $E$ into $X^{\flat\flat}$,
 but the same proof works for an arbitrary $T$, with minor modifications.
\end{proof}

We shall need the following particular case:

\begin{thm}\label{t:main-beh}
 Keep the above notation, and pick $T \in B(E, \cc X^{\flat\flat})$. Then the following are equivalent:
 \begin{enumerate}
 \item For every finite dimensional $F \subset \cc X^\flat$, and any $\vr > 0$, there exists $S \subset B(E,\cc X)$ so that $(1-\vr) \|T\| \leq \|S\| \leq (1+\vr) \|T\|$, $\langle Te, f \rangle = \langle f, Se \rangle$ for any $e \in E$ and $f \in F$, $Te = Se$ for $e \in T^{-1}(X)$, $B_j S \in (C_j)_\vr$ for any $j$.
  \item $B_j^{\flat\flat} T \in \overline{C_j}^{{\textrm{weak}^*}}$ for any $j$.
 \end{enumerate}
\end{thm}

\begin{proof}
The proof depends on the appropriate selection of the operators $A_i$ in \Cref{t:useless-beh}. 

Let $E_0 = T^{-1}(\cc X)$, and consider the restriction operator $A_0 : B(E, \cc X) \to B(E_0, \cc X)$. The condition $A_0^{\flat\flat} T = A_0 S$ translates to $T|_{E_0} = S|_{E_0}$.
 Likewise, we guarantee that $\langle Te, f \rangle = \langle f, Se \rangle$ for any $e \in E$ and $f \in F$ by taking $A_{k\ell} : B(E, \cc X) \to \R : U \mapsto \langle Ue_k, f_\ell \rangle$, where $(e_k)$ and $(f_\ell)$ are sufficiently large finite subsets of $E$ and $F$ respectively.
 It remains to show that $T$ is weak$^*$ continuous on the weak$^*$ closure of $W = \ran A_0^\flat + \sum_{k,\ell} \ran A_{k\ell}^\flat$. 
 
 Show first that $W$ itself is weak$^*$ closed. Find a projection $P$ from $E$ onto $E_0$, and let $E_1 = \ker P$; then $W = E_0 \otimes \cc X^\flat + E \otimes F = E_0 \otimes \cc X^\flat + E_1 \otimes F$. 
 By \cite[Theorem 2.5]{Oja-locref}, there exists a weak$^*$ continuous projection $Q$ from $\cc X^\flat$ to $F$. Then $P \otimes I_{X^\flat} + (I-P) \otimes Q$ is a weak$^*$ continuous projection from $E \otimes X^\flat$ onto $W$, which implies $W$ being weak$^*$ closed.

 Now use the ``standard'' principle of local reflexivity to find $R \in B(E,\cc X)$ so that $T|_{E_0} = R|_{E_0}$, and $\langle Te, f \rangle = \langle f, Re \rangle$ for any $e \in E, f \in F$. In other words, $T$ coincides with $R$ on $W$, hence the desired weak$^*$ continuity.
\end{proof}

The following result is essentially contained in \cite{Behrends} (pp.~114-115).

\begin{cor}\label{c:almost cone}
Suppose $E$ and $\cc X$ are ordered Banach spaces, with $E$ finite dimensional. If $T \in B(E,\cc X^{\flat\flat})$ satisfies $T(E^+) \subset \cc X^{\flat\flat+}$, then for every $\vr > 0$ and any finite dimensional $F \subset \cc X^\flat$ there exists $S \in B(E, \cc X)$ so that $(1-\vr) \|T\| \leq \|S\| \leq (1+\vr) \|T\|$, $\langle Te, f \rangle = \langle f, Se \rangle$ for any $e \in E$ and $f \in F$, $Te = Se$ for $e \in T^{-1}(X)$, and $S (\ball(E)^+) \subset (\cc X^+)_\vr$.
\end{cor}

\begin{proof}
We can and do assume that $\vr < 1/2$, and (by scaling) that $\|T\| = 1$. Let $(e_j)_{j=1}^N$ be an $\vr/3$-net in $\ball(E)^+$. Apply \Cref{t:main-beh} with the following approximate conditions:
\begin{itemize}
\item
$\cc Z_j = \cc X$ for $1 \leq j \leq N$; $C_j = \cc X^+$ for any $j$ (by \Cref{p:Banach Alaoglu}, $\overline{C_j}^{{\textrm{weak}^*}} = \cc X^{\flat\flat+}$);
\item
For $1 \leq j \leq N$, $B_j : B(E,\cc X) \to \cc X : U \mapsto Ue_j$ (then $B_j^{\flat\flat} T \in \cc \cc X^{\flat\flat+}$).
\end{itemize}
We can find $S \in B(E,\cc X)$ so that $(1-\vr) \|T\| \leq \|S\| \leq (1+\vr) \|T\|$, $\langle Te, f \rangle = \langle f, Se \rangle$ for any $e \in E$ and $f \in F$, and $T e_j \subset (\cc X^+)_{\vr/3}$ for any $k$.
Now pick $e \in \ball(E)^+$. Find $j$ so that $\|e - e_j\| \leq \vr/3$. Further, find $x \in \cc X^+$ so that $\|S e_j - x\| < \vr/2$. Then
$$
\|Se - x\| \leq \|S\| \|e - e_j\| + \|e_j - x\| < \Big( 1 + \frac\vr3 \Big) \frac\vr3 + \frac\vr2 < \vr .
\qedhere
$$
\end{proof}

\Cref{c:almost cone} guarantees the existence of $S$ taking the positive cone of $E$ ``almost'' to the positive cone of $X$. To map $E^+$ into $\cc X^+$, we need certain extra assumptions.

\begin{prop}\label{p:into cone}
Suppose $E$ and $\cc X$ are ordered Banach spaces, with $E$ finite dimensional and normal, and $\cc X$ -- generating.
If $T \in B(E,\cc X^{\flat\flat})$ satisfies $T(E^+) \subset \cc X^{\flat\flat+}$, then for every $\vr > 0$ and any finite dimensional $F \subset \cc X^\flat$ there exists $U \in B(E,\cc X)$ so that  $(1-\vr) \|T\| \leq \|U\| \leq (1+\vr) \|T\|$, $\big|\langle Te, f \rangle - \langle f, Ue \rangle\big| \leq \vr \|e\|  \|f\|$ for any $e \in E$ and $f \in F$, $\|Te - Ue\| \leq \vr \|e\|$ for $e \in T^{-1}(X)$, and $U(E^+) \subset \cc X^+$.
\end{prop}

The proof requires several lemmas, which may be known to experts.

\begin{lem}\label{l:positive functional}
 Suppose $E$ is a normal separable ordered Banach space. Then there exists $e^\flat \in E^{\flat+}$ so that $\langle e^\flat, e \rangle > 0$ for any $e \in E^+ \backslash \{0\}$. Consequently, if $E$ is finite dimensional, then there exists $c>0$ so that $\langle e^\flat, e \rangle \geq c \|e\|$ for any $e \in E$.
\end{lem}

\begin{proof}
By the scalar version of \Cref{nor to gen} (see also \cite{batty1984positive}), $E^{\flat+}$ is generating; find $C > 0$ exceeding the value of the generating constant.
 Let $(e_k)_{k=1}^\infty$ be a $1/(2C)$-net in $\sphere(E)^+$ (if $\dim E < \infty$, a finite net will suffice). For each $k$ find norm one $x_k^\flat \in E^\flat$ so that $\langle x_k^\flat , e_k \rangle = 1$. Further, find $y_k^\flat \in E^{\flat+}$ so that $y_k^\flat \geq \pm x_k^\flat$, and $\|y_k^\flat\| \leq C$. Note that $\langle y_k^\flat , e_k \rangle \geq 1$.
 
 We claim that $e^\flat = \sum_k 2^{-k} y_k^\flat$ has the desired properties. Indeed, for any $e \in \sphere(E)^+$ there exists $k$ so that $\|e_k - e\| \leq 1/(2C)$. Then
 $$
 \langle y_k^\flat , e \rangle \geq \langle y_k^\flat , e_k \rangle - \|y_k^\flat\| \|e - e_k\| \geq 1 - C \cdot \frac1{2C} = \frac 12 ,
 $$
 hence $\langle e^\flat, e \rangle \geq 2^{-k} \langle y_k^\flat , e \rangle > 0$. The ``consequently'' statement follows by compactness.
\end{proof}

\begin{rmk}
    By plugging $\lambda = 0$ and $\mu = -1$ into \cite[Lemma 1.1.5]{batty1984positive}, we show that, if $\cc X$ is $C$-normal, then $\cc X^\flat$ is $C$-generating. We prove similar results for matricial normed spaces in \Cref{ss:duality}.
\end{rmk}

Next we consider order boundedness of the image of the unit ball of a finite dimensional space.

\begin{lem}\label{l:order bounded}
 Suppose $E$ is finite dimensional, and $\cc X$ is a generating ordered Banach space. Then for every $V \in B(E,\cc X)$ there exists $x \in \cc X^+$ so that, for every $e \in E$, we have $-\|e\|x \leq Ve \leq \|e\|x$.
\end{lem}

\begin{proof}
 Let $(e_i)_{i=1}^n$ be an Auerbach basis in $E$. For each $i$, find $x_i \in \cc X^+$ so that $-x_i \leq Ve_i \leq x_i$. Then $x = \sum_{i=1}^n x_i$ has the desired properties. Indeed, any $e \in \ball(E)$ can be represented as $e = \sum_i \alpha_i e_i$, with $|\alpha_i| \leq 1$. Then $Ve = \sum_i \alpha_i V e_i \leq x$, since $\alpha_i V e_i \leq x_i$; the lower bound is established similarly.
\end{proof}

\begin{proof}[Proof of \Cref{p:into cone}]
For convenience, we assume that $\|T\| = 1$. Pick $C > 0$, so that $\cc X$ is $C'$-generating for some $C' \in (0,C)$. 
Find $e^\flat \in E^{\flat+}$ so that $\langle e^\flat, e \rangle \geq \|e\|$ for any $e \in E^+$ (such an $e^\flat$ exists by \Cref{l:positive functional}). Further, use \Cref{l:order bounded} to find $x \in \cc X^+$ so that $-\|e\| x \leq Te \leq \|e\| x$ for any $e \in \ball(E)$.

Let $\delta = \vr/(4 \|e^\flat\|)$, $\sigma = \vr/(4\|x\|)$, and find in $\ball(E)^+$ a $\sigma$-net $(e_k)_{k=1}^N$.
By \Cref{c:almost cone}, there exists $S \in B(E,\cc X)$ so that $1-\vr/2 < \|S\| < 1+\vr/2$ so that $Te = Se$ whenever $Te \in \cc X$, $\langle Te, f \rangle = \langle f, Se \rangle$ for any $e \in E$ and $f \in F$, and so that for any $e \in E^+$ there exists $x = x(e) \in \cc X$ so that $\|x\| \leq \delta/(NC)$, and $Se + x \in \cc X^+$. For each $k$, find $y_k \in \cc X^+$ so that $\|y_k\| \leq \delta/N$, and $y_k \geq \pm x(e_k)$ (so $Se_k + y_k \in \cc X^+$). Let $y = \sum_{k=1}^N y_k$, then $\|y\| < \delta$, and $Se_k + y \in \cc X^+$ for any $k$.

For a generic $e \in \ball(E)^+$, find $k$ so that $\|e - e_k\| \leq \sigma$, then $Se \geq Se_k - \sigma x$, and therefore, $Se + y + \sigma x \in \cc X^+$ for any $e \in \ball(E)^+$.

We claim that the operator $U$, defined by $U e = Se + \langle e^\flat, e \rangle (y + \sigma x)$ (in other words, $U = S + e^\flat \otimes (y + \sigma x)$) has the desired properties. Indeed,
$\|U - S\| \leq \|e^\flat\| (\|y\| + \sigma \|x\|) \leq \vr/2$, hence
$$
1 - \vr < \|S\| - \|U-S\| \leq \|U\| \leq \|S\| + \|U-S\| < 1 + \vr .
$$
Further, for every $e \in E, f \in F$,
$$
\big|\langle Te, f \rangle - \langle f, Ue \rangle\big| = \big|\langle f, Se \rangle - \langle f, Ue \rangle\big| \leq \|U-S\| \|e\| \|f\| \leq \vr \|e\| \|f\| .
$$
Similarly, we verify that $\|Te - Ue\| \leq \vr \|e\|$ for $e \in T^{-1}(X)$. Finally, for $e \in \sphere(E)^+$,
$$
Ue = Se + \langle e^\flat, e \rangle (y + \sigma x) \geq Se + y + \sigma x \geq 0 ,
$$
and therefore, $U(E^+) \subset \cc X^+$. 
\end{proof}

Generation properties of $\cc X^+$ need not be assumed if $E$ has special structure.

\begin{prop}\label{p:LR for lattices}
Suppose $E$ is a finite dimensional Banach lattice, and $\cc X$ is an ordered Banach space. If $T \in B(E,\cc X^{\flat\flat})$ satisfies $T(E^+) \subset \cc X^{\flat\flat+}$, then for every $\vr > 0$ and any finite dimensional $F \subset \cc X^\flat$ there exists $U \in B(E,\cc X)$ so that  $(1-\vr) \|T\| \leq \|U\| \leq (1+\vr) \|T\|$, $\big| \langle Te, f \rangle - \langle f, Ue \rangle \big| \leq \varepsilon \|e\| \|f\|$ for any $e \in E$ and $f \in F$, $\|Te - Ue\| \leq \vr \|e\|$ for $e \in T^{-1}(X)$, and $U(E^+) \subset \cc X^+$.
\end{prop}

\begin{proof}[Sketch of a proof]
 By \cite[Theorem II.3.6]{Schaefer74}, $E$ is order isomorphic to $\R^n$ (with $n = \dim E$), with its natural order structure. That is, $E$ has a normalized $1$-unconditional basis $(e_i)_{i=1}^n$ so that $e = \sum_i \alpha_i e_i \in E^+$ iff each $\alpha_i \geq 0$.
 As in the proof of \Cref{p:into cone}, assume that $\|T\| = 1$, and find $S \in B(E,\cc X)$ so that $1-\vr/2 < \|S\| < 1+\vr/2$ so that $\langle Te, f \rangle = \langle f, Se \rangle$ for any $e \in E$ and $f \in F$, and so that for any $e \in E^+$ there exists $x = x(e) \in \cc X$ so that $\|x\| \leq \varepsilon/(2N)$, and $Se + x \in \cc X^+$. 
 It is easy to check that $U \in B(E,\cc X)$, defined via $U \sum_i \alpha_i e_i = \sum_i \alpha_i (S e_i + x(e_i))$ has the desired properties.
\end{proof}

\begin{rmk}
In \Cref{p:LR for lattices}, we may not be able to find $U$ so that, in addition to other desirable properties, $\langle Te, f \rangle = \langle f, Ue \rangle$ for any $e \in E, f \in F$, and $Te = Ue$ for $e \in T^{-1}(X)$.
To show this, we re-work an example from \cite{bernau1980unified}. Denote by $\kappa$ the canonical embedding of $c_0$ into its bidual. By e.g.~\cite[Example I.1.4(d)]{HWW-M-ideals}, $\kappa$ is an M-embedding, hence $\kappa^* : \ell_1^{**} \to \ell_1$ is an L-projection.
Combining Proposition 1.2.11 and Theorem 2.4.12 of \cite{M-N}, we conclude that $\ell_1$ is a projection band in its bidual. Clearly the band projection from $\ell_1^{**}$ onto $\ell_1$ is an L-projection, and, by the uniqueness of such object \cite[Proposition I.1.2]{HWW-M-ideals}, it coincides with $\kappa^*$.

Now find fully supported positive $u \in \ell_1$ and $f_0 \in c_0$, and also a positive $v \in \ell_1^{**}$ disjoint from $\ell_1$. Find a positive $f_1 \in \ell_1^*$ so that $\langle v, f_1 \rangle > 0$.
Consider $F = \Span[\kappa f_0, f_1] \subset \ell_1^*$, and define $T : \ell_1^2 \to \ell_1^{**}$ by $T \delta_1 = u, T \delta_2 = v$ ($\delta_1, \delta_2$ is the canonical basis of $\ell_1^2$; clearly $T$ is bipositive).
Suppose, for the sake of contradiction, that there exists a positive $U : \ell_1^2 \to \ell_1$ so that $\langle Te, f \rangle = \langle f, Ue \rangle$ for any $e \in E, f \in F$, and $Te = Ue$ for $e \in T^{-1}(\ell_1)$.
In particular, we have $U \delta_1 = u$, $\langle v, \kappa f_0 \rangle = \langle \kappa f_0, U \delta_2 \rangle$, and $\langle f_1, U \delta_2 \rangle > 0$. Recalling that $\kappa^*$ is the band projection onto $\ell_1$, we conclude that $\kappa^* v = 0$, hence $\langle \kappa f_0 , v \rangle = 0$. As $f_0$ has full support, we conclude that $U \delta_2 = 0$, which clashes with $\langle f_1, U \delta_2 \rangle > 0$.
\end{rmk}

Taking $E = \R$ in \Cref{p:LR for lattices}, we obtain an ``ordered Goldstine theorem:''

\begin{cor}\label{c:goldstine from behrends}
If $X$ is an ordered Banach space, then $\ball(\cc X^{\flat\flat})^+ = \overline{\ball(\cc X)^+}^{{\textrm{weak}}^*}$. 
\end{cor}

\section{Matricial structures and matricial polars}\label{prelim}

The useful facts collected here will be used throughout the paper.

\subsection{Operator spaces: an introduction}\label{OS-intro}

An (abstract) \emph{operator} space, or a \emph{matricially normed space}, is, for us, a normed space $\cc X$, equipped with its sequence of \emph{matricial norms} $\| \cdot \|_n$ on $S_\infty^n(\cc X)$ (the space of $\cc X$-valued $n \times n$ matrices, with $n \in \N$; we identify $\cc X$ with $S_\infty^1(\cc X)$); these norms have to satisfy \emph{Ruan's axioms}:
\begin{enumerate}
    \item For any $a, b \in S_\infty^{m,n}$ and $x \in S_\infty^n(\cc X)$, $\|a \cdot x \cdot b^*\|_m \leq \|a\| \|b\| \|x\|_n$. Here, $a \cdot x \cdot b^*$ is the shorthand for $(a \otimes I_{\cc X}) \cdot x \cdot (b^* \otimes I_{\cc X})$.
    \item For any $x \in S_\infty^m(\cc X)$ and $x \in S_\infty^n(\cc X)$, $\| \begin{pmatrix} x & 0 \\ 0 & y \end{pmatrix} \|_{m+n} = \max \big\{ \|x\|_m , \|y\|_n \big\}$.
\end{enumerate}
The prototypical example of an operator space is $B(H)$ ($H$ is a Hilbert space); the matricial norms come from the identification $S_\infty^n(B(H)) = B(\ell_2^n(H))$.
Importantly, every operator space $\cc X$ has a concrete representation: there exists a Hilbert space $H$ and a map $J : \cc X \to B(H)$ so that, for every $n$, $I_{S_\infty^n} \otimes J : S_\infty^n(X) \to S_\infty^n(B(H))$ is an isometry.

For an operator $T : \cc X  \to \cc Y$ ($\cc X, \cc Y$ are operator spaces), we define its \emph{c.b. norm} $\cbnorm{T}$ as $\|I_{S_\infty} \otimes T : S_\infty(\cc X) \to S_\infty(\cc Y)\|$. 
Equivalently, $\cbnorm{T} = \sup_n \|T_n\|$, where $T_n$ stands for $I_{S_\infty^n} \otimes T$. $T$ is \emph{completely bounded} if its c.b.~norm is finite. 
For more information on operator spaces, and maps between them, we refer the reader to \cite{effros-ruan-book}, \cite{paulsen_book}, or \cite{Pisier-Op-Spaces}.

If $\cc X$ is a vector space, and $\cc X^\flat$ is its dual, then, for $x = (x_{ij}) \in S_\infty^m(\cc X)$ and $x = (x^\flat_{k\ell}) \in S_\infty^n(\cc X)$, we define $\mp{x^\flat}{x} = \big( \langle x_{k\ell}^\flat, x_{ij} \rangle \big) \in S_\infty^{mn}$ (the rows and the columns of the resulting $mn \times mn$ matrix are indexed by pairs $(i,k)$ and $(j,\ell)$, respectively).
This permits us to view $x^\flat$ as an operator $\cc X \to S_\infty^n$. If $\cc X$ is equipped with matricial norms, we define $\|x^\flat\|$ as the c.b.~norm of the above operator, or in other words, $\|x^\flat\| = \sup \big\{ \|\mp{x^\flat}{x}\|_{nm} : \|x\|_m \leq 1\big\}$; this expression turns out to equal $\sup \big\{ \|\mp{x^\flat}{x}\|_{nm} : \|x\|_m \leq 1\big\}$ (Smith's Lemma, see e.g.~\cite[Proposition 8.11]{paulsen_book}). 

\subsection{Matricial polars}\label{matricial duality}

Duality between matricially normed spaces was investigated in \cite{effros1997matrix}. We recall some of the important results here, and refer the reader to the original paper for more detail.

If $\cc X$ is a vector space, then $K = (K_n)$ is a \emph{matricial convex set on $\cc X$} if, for every $n$, $K_n \subset S_\infty^n(\cc X)$, and (i) for any $m, n$, $K_m \oplus K_n \subset K_{m+n}$, (ii) $\alpha^* \cdot K_n \cdot \alpha \subset K_m$ for any $\alpha \in \ball(S_\infty^{n,m})$. Clearly, in this setting, each $K_n$ is convex, and contains $0$.

Note that our definition differs slightly from the one given in \cite{effros1997matrix}, where the inclusion $\alpha^* \cdot K_n \cdot \alpha \subset K_m$ was only required for $\|\alpha\| = 1$. However, it is easy to see that a matricial convex sets ``in the original sense'' $(K_n)$ satisfies $0 \in K_1$ iff $(K_n)$ has the properties listed in the preceding paragraph.

Suppose $(\cc X, \cc X^\flat)$ is a dual pair of vector spaces, and $K = (K_n)$ is a matricial convex set on $\cc X$. Define the \emph{matricial polar} of $K$, denoted by $K^\pi = (K_n^\pi)$:
$x^\flat \in S_\infty^n(\cc X^\flat)$ belongs to $K_n^\pi$ iff $\Re \mp{x^\flat}{x} \leq I_{mn}$ (the identity on $\ell_2^{mn}$) for any $x \in K_n$. 
Actually, checking $m=n$ suffices:
\begin{prop}
\label{control dimension}
In the above notation, $x^\flat \in S_\infty^n(\cc X^\flat)$ belongs to $K_n^\pi$ iff $\Re \mp{x^\flat}{x} \leq I_{S_\infty^{n^2}}$ for any $x \in K_n$. 
\end{prop}

Further, we shall say $K = (K_n)$ is \emph{weakly-closed} or $\sigma(\mathcal X, \mathcal X^\flat)$-closed if for each $n \in \mathbb N$, if $(x_i)_{i \in I} \subset K_n$ such that $\sp{x^\flat}{x_i} \to \sp{x^\flat}{x}$ for every $x^\flat \in S_\infty^n(\mathcal X^\flat)$, then $x \in K_n$.

\begin{thm}[Separation theorem]
\label{separation}
    Suppose $\cc X$ and $\cc X^\flat$ are as above, $K = (K_n)$ is weakly closed, and matricial convex on $\cc X$, and $x \in S_\infty^n(\cc X)$ doesn't belong to $K_n$.
    Then there exists $x^\flat \in K_n^\pi$ so that $\Re \mp{x^\flat}{x} \not\leq I_{n^2}$.
\end{thm}

\section{Matricial order operator spaces: definitions}\label{s: basic properties}

In the 1970s, operator systems (that is, self-adjoint subspaces of $B(H)$, for some Hilbert space $H$) appeared in the literature (such as e.g.~\cite{choi1977injectivity}) as a natural generalization of the GNS construction for $C^*$-algebras.
As with operator spaces, for an operator system $\cc X$ we identify $S_\infty^n(\cc X)$ with a subspace of $B(\ell_2^n(H))$, with the inherited order. This is an elementary example of a matricial order space, which we define in the following manner:

\begin{defn}\label{star vector space}
    A \emph{$*$-vector space} is a complex topological vector space $\cc X$, equipped with a continuous involution $*$. If our space is normed, then the involution is assumed to be isometric.
    The \emph{Hermitian} part of $\cc X$, denoted by $\cc X_h$, is the real-linear subspace of $\cc X$, consisting of elements invariant under involution.
\end{defn}


Following \cite{humeniuk2023extension}, and earlier works \cite{Schreiner-JFA, Schreiner-DISS}, we consider \emph{matricial order structures} on $*$-vector spaces: 

\begin{defn}\label{d:ordered properties}
A \emph{matricial order space} is a pair $(\cc X, \{S_\infty^n(\cc X)^+\}_{n \in \bb N})$ where $\cc X$ is a $*$-vector space, and the \emph{positive wedge} $\{S_\infty^n(\cc X)^+\}_{n \in \bb N}$ satisfies:
\begin{enumerate}
    \item
    For any $n \in \N$, $S_\infty^n(\cc X)^+$ is closed, and is a wedge in $S_\infty^n(\cc X)^+$. 
    \item
    For any $m,n \in \N$, $x \in S_\infty^n(\cc X)^+$, and $a \in S_\infty^{m,n}$, $a \cdot x \cdot a^* \in S_\infty^m(\cc X)^+$.
    \item 
    For any $m,n \in \N$, $x \in S_\infty^m(\cc X)^+$ and $y \in S_\infty^n(\cc X)^+$, $x \oplus y := \begin{pmatrix} x & 0 \\ 0 & y \end{pmatrix} \in S_\infty^{m+n}(\cc X)^+$.
    \item
    For any $n$, $S_\infty^n(\cc X)^+ \subset S_\infty^n(\cc X)_h$ (the set of \emph{hermitian}, or \emph{self-adjoint}, elements of $S_\infty^n(\cc X)$).
\end{enumerate}
If, in addition, $\a$ is an operator space semi-norm (that is, a sequence of semi-norms $(\alpha_n)_{n \in \N}$, satisfying Ruan's axioms), we call a triple $(\cc X, \{S_\infty^n(\cc X)^+\}_{n \in \bb N}, \a)$  a \emph{matricial order operator space.}
\end{defn}
Here and throughout this paper, we define the \emph{adjoint} of $x = \sum_{i,j} E_{ij} \otimes x_{ij} \in S_\infty^n(\cc X)$ as $x^* = \sum_{i,j} E_{ij} \otimes x_{ji}^*$, with $E_{ij}$ being matrix units. In other words, $(a \otimes x)^* = a^* \otimes x^*$. We assume that $\|x\| = \|x^*\|$ for any $x \in S_\infty^n(\cc X)$. We shall need a simple linear algebra result.

\begin{lem}\label{lem:hermitian}
For $n \in \N$ and $\cc X$ as above, $S_\infty^n(\cc X)_h$ coincides, as a real vector space, with $[S_\infty^n]_h \otimes \cc X_h$.
\end{lem}

\begin{proof}
Suppose $x = x^* \in S_\infty^n(\cc X)_h$. Write $x = \sum_k a_k \otimes x_k$, with $a_k \in S_\infty^n, x_k \in \cc X$. Split both $a_k$ and $x_k$ into their real and imaginary parts. That is, write $a_k = b_k + \iota c_k$, with 
$$
b_k = \frac{a_k+a_k^*}2 , \, \, c_k = \frac{a_k - a_k^*}{2i}
$$
(here $b_k = b_k^*, c_k = c_k^*$). Likewise, write $x_k = y_k + \iota z_k$. Then
\begin{align*}
    a_k \otimes x_k & = \big( b_k \otimes y_k - c_k \otimes z_k \big) + \iota \big( b_k \otimes z_k + c_k \otimes y_k \big) , \\
    (a_k \otimes x_k)^* & = \big( b_k \otimes y_k - c_k \otimes z_k \big) - \iota \big( b_k \otimes z_k + c_k \otimes y_k \big) .
\end{align*}
Therefore,
$$
x = \frac{x+x^*}2 = \sum_k \frac{a_k \otimes x_k + (a_k \otimes x_k )^*}2 = \sum_k \big( b_k \otimes y_k - c_k \otimes z_k \big) . \qedhere
$$
\end{proof}

In the context of \Cref{d:ordered properties}, we say that $S_\infty^n(\cc X)^+$ is the \emph{positive cone} (and not merely a wedge) if $-S_\infty^n(\cc X)^+ \cap S_\infty^n(\cc X)^+ = \{0\}$. It turns out that being a cone is determined ``on the lowest level.''

\begin{lem}\label{lem:when cone}
If $\cc X^+$ is a cone, then $S_\infty^n(\cc X)^+$ is a cone for every $n$.
\end{lem}

\begin{proof}
Suppose $x = (x_{ij})_{i,j=1}^n \in S_\infty^n(\cc X)$ is such that $x, -x \in S_\infty^n(\cc X)^+$. Then for any $a \in \ell_2^n$, $\pm a^* \cdot x \cdot a \in \cc X^+$ (here we view $a$ as an $n \times 1$ matrix).

Fixing $i \in \{1, \ldots, n\}$, and taking $a = |i\rangle$, we conclude that $\pm x_{ii} \in \cc X^+$, hence $x_{ii} = 0$,  for any $i$.
Next, for $i \neq j$, take $a = |i\rangle + |j\rangle$ and $|i\rangle + \iota |j\rangle$. This gives us $\pm(x_{ij} + x_{ij}^*), \pm \iota (x_{ij} - x_{ij}^*) \in \cc X^+$. This can only happen when $x_{ij} + x_{ij}^* = 0 = x_{ij} - x_{ij}^*$, so $x_{ij} = 0$.
\end{proof}

\begin{rmk}
As noted above, any operator system (in particular, any $C^*$-algebra) is a matricial order operator space; we further investigate these examples in \Cref{ss:examples}.
While operator systems can be realized as unital self-adjoint subspaces of $B(H)$, this is false for general matricial order operator spaces, see \cite{werner2002matrix} and \cite{debruyn2025archimedean}. 
\end{rmk}

We also need to define positivity properties of maps between matricial order spaces:

\begin{defn}\label{bipositivity}
Suppose $\cc X, \cc Y$ are matricial order spaces. Then an involution-preserving map $T : \cc X \to \cc Y$ is called \emph{completely positive} if $I_{S_\infty^n} \otimes T : S_\infty^n(\cc X) \to S_\infty^n(\cc Y)$ is positive for every $n$ -- that is, the positivity of $x \in S_\infty^n(\cc X)$ implies that of $I_{S_\infty^n} \otimes T(x)$.
If $x \in S_\infty^n(\cc X)^+$ if and only if $I_{S_\infty^n} \otimes T(x) \in S_\infty^n(\cc Y)^+$ for every $n$, we say that $T$ is \emph{completely bipositive}.
\end{defn}

It is easy to see that, if $\cc X_h$ is generating, and $T : \cc X \to \cc Y$ satisfies $T (\cc X^+) \subset \cc Y^+$ for any $n$, then $T$ is automatically involution-preserving.

A classical result states that, for an operator space $\cc X$, any $x^\flat \in \cc X^\flat$ is completely bounded, with $\|x^\flat\| = \|x^\flat\|_{cb}$. Reasoning as in \cite[Proposition 3.8]{paulsen_book}, we easily establish:

\begin{prop}
If $\cc X$ is a matricial normed space, $x^\flat \in \cc X^\flat$ is completely positive if and only if it is positive.
\end{prop}


Often, positive cones of our matricial order operator spaces have certain convenient properties:

\begin{defn}\label{d:special ordered properties}
A matricial order operator space $(\cc X, \{S_\infty^n(\cc X)^+\}_{n \in \bb N}, \a)$ is said to be:
\begin{enumerate}
\item $C$-\emph{matricial normal}: 
if $\begin{pmatrix} u_1 & u \\ u^* & u_2 \end{pmatrix} \in S_\infty^{2n}(\cc X)^+$, then $\a_n(u) \leq C \a_n(u_1) \vee \a_n(u_2).$
\item $C$-\emph{matricial generating}:  for any $\vr > 0$, and any $u \in S_\infty^n(\cc X)$, there exist $u_1, u_2 \in S_\infty^n(\cc X)^+$ such that $\begin{pmatrix} u_1 & u \\ u^* & u_2 \end{pmatrix} \in S_\infty^{2n}(\cc X)^+$, and $\a_n(u_1) \vee \a_n(u_2) \leq C \a_n(u) + \e $. If the above condition also holds for $\vr = 0$, we shall say that our cone is $C$-\emph{strongly matricial generating}.
\item \emph{matricial normal} (\emph{matricial generating}) if it is $C$-matricial normal (resp.~$C$-matricial generating) for some $C$.
\end{enumerate}
\end{defn}

As in \Cref{s:order Banach}, we see that, if $\cc X$ is matricial normal, then $S_\infty^n(\cc X)^+$ is a cone, for every $n$.

\begin{rmk}\label{about selfadjointness}
The preceding definition is motivated by normality and generation of ordered Banach spaces, introduced in \Cref{s:order Banach}. To see the connection, suppose $(\cc X, \| \cdot \|)$ is an ordered normed space. Observe first that, if $-y \leq x = x^* \leq y$, then $\begin{pmatrix} y & x \\ x & y \end{pmatrix} \geq 0$. Indeed, $y - x \geq 0$, hence
\[
\begin{pmatrix}  1 & -1 \\ -1 & 1 \end{pmatrix} \otimes (y-x) = \begin{pmatrix}  y-x & x-y \\ x-y & y-x \end{pmatrix} \geq 0 .
\]
Likewise, $y + x \geq 0$, hence
\[
\begin{pmatrix}  1 & 1 \\ 1 & 1 \end{pmatrix} \otimes (y+x) = \begin{pmatrix}  y+x & x+y \\ x+y & y+x \end{pmatrix} \geq 0 .
\]
Averaging the two centered expressions, we obtain $\begin{pmatrix} y & x \\ x & y \end{pmatrix} \geq 0$.

Now suppose $x$ is self-adjoint, and $\begin{pmatrix} y_1 & x \\ x & y_2 \end{pmatrix} \geq 0$. Letting $y = (y_1 + y_2)/2$, we obtain
\[
\begin{pmatrix} 1 \\ 1 \end{pmatrix}^* \begin{pmatrix} y_1 & x \\ x & y_2 \end{pmatrix} \begin{pmatrix} 1 \\ 1 \end{pmatrix} = 2(y+x) \geq 0 .
\]
Also,
\[
\begin{pmatrix} 1 \\ -1 \end{pmatrix}^* \begin{pmatrix} y_1 & x \\ x & y_2 \end{pmatrix} \begin{pmatrix} 1 \\ -1 \end{pmatrix} = 2(y-x) \geq 0 .
\]
Thus, $-y \leq x \leq y$, and $\|y\| \leq \|y_1\| \vee \|y_2\|$
\end{rmk}

\begin{rmk}
The above reasoning shows that, for a matricial generating $\cc X$, $x \in S_\infty^n(\cc X)$ is hermitian if and only if it is a difference of two positive elements. Indeed, any positive element is hermitian, hence real linear combinations of such are hermitian as well. On the other hand, if $x = x^* \in S_\infty^n(\cc X)$, then matricial generation implies the existence of $y \in S_\infty^n(\cc X)^+$ so that $y \geq x \geq -y$. Then $x = (y+x)/2 - (y-x)/2$; both terms are positive.
\end{rmk}

\begin{lem}\label{check-selfadj}
Suppose $\cc X$ is a matricial order operator space.

$(1)$ $\cc X$ is $C$-matricial normal if and only if, for any $u \in S_\infty^n(\cc X)_h$ and $v \in S_\infty^n(\cc X)^+$ with $v \geq \pm u$, we have $\alpha_n(u) \leq C \alpha_n(v)$.

$(2)$ $\cc X$ is $C$-matricial generating if and only if, for any $u \in S_\infty^n(\cc X)_h$, and any $\vr > 0$, we can find $v \in S_\infty^n(\cc X)^+$ with $v \geq \pm u$, we have $\alpha_n(v) < C\alpha_n(u) + \vr$.
\end{lem}

\begin{proof}
(1) If $\cc X$ is $C$-matricial normal, then, by \Cref{about selfadjointness}, for any $u \in S_\infty^n(\cc X)_h$ and $v \in S_\infty^n(\cc X)^+$ with $v \geq \pm u$, we have $\alpha_n(u) \leq C \alpha_n(v)$.

It remains to establish the converse. Suppose $u \in S_\infty^n(\cc X)$ is such that $\begin{pmatrix}  x_1 & u  \\  u^* & x_2  \end{pmatrix} \in S_\infty^{2n}(\cc X)^+$, and $\alpha_n(x_i) \leq 1$ for $i=1,2$. We have to show that $\alpha_n(u) \leq C$.

First note that $\overline{u} = \begin{pmatrix}  0 & u  \\  u^* & 0  \end{pmatrix}$ is self-adjoint, and $\alpha_n(u) = \alpha_{2n}(\overline{u})$. Likewise, $\overline{x} = \begin{pmatrix}  x_1 & 0  \\  0 & x_2  \end{pmatrix}$ is positive, and $\vee_i \alpha_n(x_i) = \alpha_{2n}(\overline{x})$. A simple computation shows that
\begin{align*}
&
\begin{pmatrix}  \overline{x} & \overline{u}  \\  \overline{u} & \overline{x}  \end{pmatrix}  =
\begin{pmatrix}  x_1 & 0 & 0 & u  \\  0 & x_2 & u^* & 0  \\  0 & u & x_1 & 0  \\   u^* & 0 & 0 & x_2   \end{pmatrix}  
\\
&
= 
\begin{pmatrix}  1 & 0  \\  0 & 0  \\  0 & 0  \\   0 & 1   \end{pmatrix}  \begin{pmatrix}  x_1 & u  \\  u^* & x_2  \end{pmatrix} \begin{pmatrix}  1 & 0 & 0 & 0  \\  0 & 0 & 0 & 1   \end{pmatrix}  +
\begin{pmatrix}  0 & 0  \\  0 & 1  \\  1 & 0  \\   0 & 0   \end{pmatrix}  \begin{pmatrix}  x_1 & u  \\  u^* & x_2  \end{pmatrix} \begin{pmatrix}  0 & 0 & 1 & 0  \\  0 & 1 & 0 & 0   \end{pmatrix} \geq 0 ,
\end{align*}
By \Cref{about selfadjointness} again, $\overline{x} \geq \pm \overline{u}$, which leads to $C\alpha_{2n}(\overline{x}) \geq \alpha_{2n}(\overline{u})$. In light of our earlier discussion on $\alpha_{2n}(\overline{x})$ and $\alpha_{2n}(\overline{u})$, we are done.

(2) If $\cc X$ is $C$-matricial generating, then \Cref{about selfadjointness} implies that for any $u \in S_\infty^n(\cc X)_h$ and any $\vr > 0$ there exists $v \in S_\infty^n(\cc X)^+$ so that $v \geq \pm u$ and $\alpha_n(v) < C \alpha_n(u) + \vr$.

Conversely, suppose for any self-adjoint $u$ and positive $\vr$ there exists $v$ with the above properties. Fix $u \in S_\infty^n(\cc X)$ and $\vr > 0$; our goal is to find $x_1, x_2 \in S_\infty^n(\cc X)^+$ so that $\vee_i \alpha_n(x_i) < C \alpha_n(u) + \vr$, and $\begin{pmatrix}  x_1 & u  \\  u^* & x_2  \end{pmatrix}  \in S_\infty^{2n}(\cc X)^+$.

As in part (1), let $\overline{u} = \begin{pmatrix}  0 & u  \\  u^* & 0  \end{pmatrix}$, and note that $\alpha_{2n}(\overline{u}) = \alpha_n(u)$. Find $v \in S_\infty^{2n}(\cc X)^+$ so that $v \geq \pm \overline{u}$, and $\alpha_{2n}(v) < C \alpha_{2n}(\overline{u}) + \vr = C \alpha_n(u) + \vr$.

Write $v = \begin{pmatrix}  x_1 & x  \\  x^* & x_2 \end{pmatrix}$. By \Cref{about selfadjointness},
$$
\begin{pmatrix} 
x_1 & x & 0 & u \\
x^* & x_2 & u^* & 0 \\
0 & u & x_1 & x \\
u^* & 0 & x^* & x_2 
\end{pmatrix} \in S_\infty^{4n}(\cc X)^+ .
$$
Consequently,
$$
\begin{pmatrix}    x_1 & u   \\  u^* & x_2  \end{pmatrix}  =
\begin{pmatrix}    1 & 0 & 0 & 0  \\  0 & 0 & 0 & 1  \end{pmatrix}
\begin{pmatrix} 
x_1 & x & 0 & u \\
x^* & x_2 & u^* & 0 \\
0 & u & x_1 & x \\
u^* & 0 & x^* & x_2 
\end{pmatrix} 
\begin{pmatrix}    1 & 0 \\ 0 & 0  \\  0 & 0 \\ 0 & 1  \end{pmatrix}
\in S_\infty^{2n}(\cc X)^+ .
$$
To conclude the proof, note that
$$
\alpha_n(x_1) \vee \alpha_n(x_2) \leq \alpha_{2n} \Big( \begin{pmatrix}    x_1 & u   \\  u^* & x_2  \end{pmatrix}   \Big) < C \alpha_n(u) + \vr .  \qedhere
$$
\end{proof}

The following result re-states matricial normality and generation.

\begin{lem}\label{reformulate normality}
    Suppose $\cc X$ is a matricial order operator space.
    \begin{enumerate}
    \item
    $\cc X$ is $C$-matricial normal if and only if, for any $a, b \in S_\infty^n({\mathcal X})^+$, we have $C \alpha_n(a+b) \geq \alpha_n(a-b)$.
    \item 
    $\cc X$ is $C$-matricial generating if and only if, for any $\vr > 0$ and $x \in S_\infty^n({\mathcal X})_h$, there exist $a, b \in S_\infty^n({\mathcal X})^+$ so that $x=a-b$ and $\alpha_n(a+b) \leq C \alpha_n(x) + \vr$.
    \end{enumerate}
\end{lem}

\begin{proof}
We only prove (1), as (2) is similar. By \Cref{check-selfadj}, it suffices to check for normality on self-adjoint elements.

(i) Suppose the above formula holds. To show that $y \geq x \geq -y$ implies $C \alpha_n(y) \geq \alpha_n(x)$, apply it with $a = (y+x)/2, b = (y-x)/2$; then $y = a+b, x = a-b$.

(ii) Conversely, suppose $\cc X$ is $C$-matricial normal. For positive $a, b$, take $y = a+b, x = a-b$. Then $y \geq x \geq -y$, hence $C \alpha_n(a+b) \geq \alpha_n(a-b)$.
\end{proof}

We show that, for operators between normal and generating spaces, the c.b.~norm can be calculated on the positive cone.
More specifically, for an operator $u : \cc X \to \cc Y$ (where $\cc X, \cc Y$ are normed matricial order spaces), define $\|u\|_{cbc}$ (the \emph{c.b.~norm on the positive cone)} as $\sup \big\{ \|u_n x\| : x \in \ball(S_\infty^n(\cc X))^+ \big\}$. Clearly, $\|u\|_{cbc} \leq \|u\|_{cb}$. We show that the converse is also true if the spaces in question possess certain properties.

\begin{lem}\label{compute norm on cone}
    Suppose the matricial order operator spaces $(\cc X, \alpha_n)$ and $(\cc Y, \beta_n)$ are $C_1$-matricial generating and $C_2$-matricial normal, respectively. Then a completely positive $u \in B(\cc X: \cc Y)$ is completely bounded iff $\|u\|_{cbc} < \infty$; in this case, $\|u\|_{cbc} \leq \|u\|_{cb} \leq C_1 C_2 \|u\|_{cbc}$.
\end{lem}

\begin{proof}
    The inequality $\|u\|_{cbc} \leq \|u\|_{cb}$ is straightforward. For the converse, we need to show that, if $\|u\|_{cbc} \leq 1$, and $x \in S_\infty^n(\cc X)$ satisfies $\alpha_n(x) <  1$, then $\beta_n(u_n x) < C_1 C_2$.

    As $\cc X$ is $C_1$-matricially generating, we can find $x_1, x_2 \in S_\infty^n(\cc X)^+$ so that $\alpha_n(x_1), \alpha_n(x_2) < C_1$, and $\begin{pmatrix}        x_1 & x  \\  x^* & x_2    \end{pmatrix} \in S_\infty^{2n}(\cc X)^+$.
    Then $\begin{pmatrix}        u_n x_1 & u_n x  \\  (u_n x)^* & u_n x_2    \end{pmatrix} \in S_\infty^{2n}(\cc Y)^+$ (recall that $u$ is completely positive, hence in particular $u_n x^* = (u_n x)^*$.
    Further, $\beta_n(u_n x_1), \beta_n(u_n x_2) < C_1$, hence, by the $C_2$-matricial normality of $\cc Y$, $\beta_n(u_n x) < C_1 C_2$.
\end{proof}

We briefly address subspaces and quotients of matricial order spaces.

It is easy to see that a subspace of a matricial order space inherits the matricial order. Generation may be lost (a subspace may have trivial intersection with the positve cone), but normality is preserved.

To define quotients, we follow \cite{humeniuk2023extension}. A subspace $J \subset \cc X$ (we assume $\cc X^+$ is a cone) is  called a \emph{kernel} if it is a kernel of a bounded completely positive map into a matricial order operator space $\cc Y$, where $\cc Y^+$ is a cone. The matricial norms on $S_\infty^n(\cc X/J)$ are defined in the usual manner.
As $J = J^*$, we can define involution on $\cc X/J$ via $(x+J)^* = x^*+J$. Let $S_\infty^n(\cc X/J)^+$ be the closure of $q_n (S_\infty^n(\cc X)^+)$, where $q : \cc X \to \cc X/J$ is the quotient map. By \cite[Section 3]{humeniuk2023extension}, if $S_\infty^n(\cc X)^+$ is a cone, then so is $S_\infty^n(\cc X/J)^+$.

\begin{prop}\label{quotient-generating}
    In the above notation, if $\cc X$ is $C$-matricial generating, then so is $\cc X/J$.
\end{prop}

\begin{proof}
It suffices to show that, if $u \in S_\infty^n(\cc X/J)_h$ satisfies $\|u\| < 1$, then there exists $v \in S_\infty^n(\cc X/J)_h$ so that $\|v\| \leq C$, and $\pm u \leq v$.

We can represent $u = x + S_\infty^n(J)$, with $x \in S_\infty^n(\cc X), \|x\| < 1$. As the involution preserves $J$, we also have $u = x^* + S_\infty^n(J)$, hence $u = x_0 + S_\infty^n(J)$, where $x_0 = (x+x^*)/2$ is self-adjoint, and $\|x_0\| < 1$. Then there exists $y_0 \in S_\infty^n(\cc X)^+$ so that $\pm x_0 \leq y_0$ and $\|y_0\| < C$. Clearly, $v = y_0 + J$ works for us.
\end{proof}

\begin{rmk}
For ordered Banach spaces, normality constants aren't preserved by taking quotients. Let $\cc X = \R^3$ with the $\ell_2$-norm, and with $\cc X^+ = \{(x,y,z) : z \geq |x|, y \geq 0\}$. One can show that $\cc X$ is $1$-normal. 
Now let $J$ be the $z$-axis, and let $q : \cc X \to \cc X/J$ be the quotient map. Then $\cc X/J$ is $\ell_2^2$, and $(\cc X/J)^+$ is the wedge $\{(x,y) : y \geq 0\}$; it is not even a cone, so no normality is possible.
\end{rmk}

Finally, we mention that \cite{debruyn2025archimedean} develops a theory of unitizations of normal matricial order spaces (with a slightly different definition of the normality constant).

\section{Positivisation in matricial order spaces}\label{positivisation}

Now suppose a matricial order space $\cc X$ is equipped with a family of semi-norms $\alpha = (\alpha_n)$. Further, assume that the matricial order on $\cc X$ is \emph{majorizing}: given any $x \in S_\infty^n(\cc X)$, there exists $x_i \in S_\infty^n(\cc X)^+$ such that $\begin{pmatrix} x_1 & x \\ x^* & x_2 \end{pmatrix} \in S_\infty^{2n}(\cc X)^+$. 
Inspired by a construction from \cite{davies1968predual}, define a new sequence of semi-norms, consistent with the order: for $x \in S_\infty^n(\cc X)$, let
\begin{equation}
\alpha^+_n(x) := \inf \Big\{ \alpha_n(x_1) \vee \alpha_n(x_2) : \begin{pmatrix}  x_ 1 & x  \\  x^* & x_2  \end{pmatrix}  \in S_\infty^{2n}(\cc X)^+ \Big\} .    
\label{eq:alpha+}
\end{equation}
Note that, for every $n$, $\alpha_n^+$ is a semi-norm, and, for $x \in S_\infty^n(\cc X)^+$, we necessarily have $\alpha_n^+(x) \leq \alpha_n(x)$. 
If $(\alpha_n)$ turns $\cc X$ into a matricial generating, matricial normal space, then $\alpha_n^+ = \alpha_n$. Also, reasoning as before, one can show that, for $x \in S_\infty^n(\cc X)_h$, $\alpha_n^+(x) = \inf \big\{ \alpha_n(y) : y \geq \pm x \big\}$.

The most essential properties of $(\alpha_n^+)$ are given below.

\begin{prop}\label{improve norm}
    $(1)$    
    $(\alpha_n^+)$ satisfy Ruan's axioms.

    $(2)$     The family of semi-norms $(\alpha_n^+)$ on $\cc X$ is $1$-matricial normal and $1$-matricial generating.

    $(3)$     For every $n$, $\alpha_n^+ = (\alpha^+)_n^+$.
\end{prop}

\begin{proof}
(1) We need to show that 
\begin{enumerate}[(i)]
\item 
For any $x \in S_\infty^n(\cc X)$ and $a, b \in S_\infty^{n,m}$, we have $\alpha_m^+(a^* \cdot x \cdot b) \leq \|a\| \alpha_n^+(x) \|b\|$.
\item 
For any $x \in S_\infty^n(\cc X)$ and $y \in S_\infty^m(\cc X)$, $\alpha_{n+m}^+(x \oplus y) = \alpha_n^+(x) \vee \alpha_m^+(y)$; here $x \oplus y = \begin{pmatrix} x & 0 \\ 0 & y \end{pmatrix}$.
\end{enumerate}

For (i), we can assume (by replacing $a$ and $b$ by $ta$ and $t^{-1} b$ if necessary) that $\|a\|=\|b\|$. Fix $\vr > 0$, and find $x_1, x_2$ so that $\begin{pmatrix} x_1 & x \\ x^* & x_2 \end{pmatrix} \in S_\infty^{2n}(\cc X)^+$, and $\alpha_n(x_1), \alpha_n(x_2) < \alpha_n^+(x) + \vr$. Then 
$$
\begin{pmatrix} a^* x_1 a & a^* x b \\ (a^* x b)^* & b^* x_2 b \end{pmatrix} =
\begin{pmatrix} a & 0 \\ 0 & b \end{pmatrix}^*
\begin{pmatrix} x_1 & x \\ x^* & x_2 \end{pmatrix}
\begin{pmatrix} a & 0 \\ 0 & b \end{pmatrix} \in S_\infty^{2}(\cc X)^+ ,
$$
By Ruan's axioms, $\alpha_m(a^* \cdot x_1 \cdot b) \leq \|a\|^2 \alpha_n(x_1)$, and $\alpha_m(b^* \cdot x_2 \cdot b) \leq \|b\|^2 \alpha_n(x_2)$; recalling $\|a\| = \|b\|$, we obtain $\alpha_m^+(a^* \cdot x \cdot b) \leq \|a\| \|b\| \alpha_n^+(x)$.

To establish (ii), fix $\vr > 0$, and find $x_1, x_2, y_1, y_2$ so that
$$
\begin{pmatrix} x_1 & x \\ x^* & x_2 \end{pmatrix} \in S_\infty^{2n}(\cc X)^+, \begin{pmatrix} y_1 & y \\ y^* & y_2 \end{pmatrix} \in S_\infty^{2m}(\cc X)^+ , \alpha_n(x_i) < \alpha_n^+(x) + \vr, \alpha_m(y_i) < \alpha_m^+(y) + \vr (i=1,2) .
$$
Then
$$
\begin{pmatrix} x_1 & 0 & x & 0 \\ 0 & y_1 & 0 & y \\ x^* & 0 & x_2 & 0 \\ 0 & y^* & 0 & y_2 \end{pmatrix} = \begin{pmatrix} x_1 & x \\ x^* & x_2 \end{pmatrix} \otimes \begin{pmatrix} 1 & 0 \\ 0 & 0 \end{pmatrix} + \begin{pmatrix} y_1 & y \\ y^* & y_2 \end{pmatrix} \otimes \begin{pmatrix} 0 & 0 \\ 0 & 1 \end{pmatrix} \in S_\infty^{2(n+m)}(\cc X)^+ ,
$$
hence
$$
\alpha_{n+m}^+ \Big( \begin{pmatrix} x & 0 \\ 0 & y \end{pmatrix} \Big) \leq \vee_{i=1,2} \alpha_{n+m} \Big( \begin{pmatrix} x_i & 0 \\ 0 & y_i \end{pmatrix} \Big) < \alpha_n^+(x) \vee \alpha_m^+(y) + \vr .
$$

(2) We establish matricial normality, since generation is similar. In light of \Cref{check-selfadj}, it suffices to show that, for every $x \in S_\infty^n(\cc X)_h$ and $y \in S_\infty^n(\cc X)^+$, with $y \geq \pm x$ we have $\alpha_n^+(x) \leq \alpha_n^+(y)$.
By the preceding reasoning, $\alpha_n^+(y) = \inf\{ \alpha_n(z) : z \geq y\}$. But any $z$ as above satisfies $\pm x \leq z$, hence $\alpha_n^+(x) \leq \inf\{ \alpha_n(z) : z \geq y\} = \alpha_n^+(y)$.

(3) For a $x \in S_\infty^n(\cc X)$,
$$
\alpha^+_n(x) = \inf \Big\{ \alpha_n(x_1) \vee \alpha_n(x_2) : \begin{pmatrix}  x_ 1 & x  \\  x^* & x_2  \end{pmatrix}  \in S_\infty^{2n}(\cc X)^+ \Big\} .
$$
Clearly $x_1, x_2 \in S_\infty^n(\cc X)^+$, hence $\alpha_n^+(x_i) \leq \alpha_n(x_i)$, for $i = 1,2$. Therefore,
$$
(\alpha^+)^+_n(x) = \inf \Big\{ \alpha_n^+(x_1) \vee \alpha^+_n(x_2) : \begin{pmatrix}  x_1 & x  \\  x^* & x_2  \end{pmatrix}  \in S_\infty^{2n}(\cc X)^+ \Big\} \leq \alpha^+_n(x) .
$$

Conversely, suppose $(\alpha^+)^+_n(x) < 1$. Find $x_1, x_2$ so that $\vee_i \alpha_n^+(x_i) < 1$, and $\begin{pmatrix}  x_1 & x  \\  x^* & x_2  \end{pmatrix}  \in S_\infty^{2n}(\cc X)^+$. For $i=1,2$ find $y_i \geq x_i$ so that $\alpha_n(y_i) < 1$. Then
$$
\begin{pmatrix}  y_1 & x  \\  x^* & y_2  \end{pmatrix} = \begin{pmatrix}  x_1 & x  \\  x^* & x_2  \end{pmatrix} + \begin{pmatrix}  y_1-x_1 & 0  \\  0 & y_2-x_2  \end{pmatrix} \in S_\infty^{2n}(\cc X)^+ ,
$$  
hence $\alpha_n^+(x) \leq \vee_i \alpha_n(y_i) < 1$.
\end{proof}



From the above, it is easy to conclude that generating and normal spaces can be renormed to be $1$-generating and $1$-normal.

\begin{prop}\label{p:renorm}
    For a matricial order operator space $({\cc X}, (\alpha_n))$, define $\alpha_n^+$ as in \eqref{eq:alpha+}. Then: 
    \begin{enumerate}
        \item If $(\alpha_n)$ is $C_g$-matricial generating, then $\alpha_n^+ \leq C_g \alpha_n$.
        \item If $(\alpha_n)$ is $C_n$-matricial normal, then $\alpha_n^+ \geq \alpha_n/C_n$.
    \end{enumerate}
\end{prop}

\section{Duality of matricial order spaces}\label{ss:duality}

Suppose $\cc X$ is a matricial order space, and $\cc X^\flat$ is its topological vector space dual (for instance, $\cc X^\flat$ can be the space of norm continuous functionals on $\cc X$, or the space of weak$^*$-continuous functionals, that is, the pre-dual of $\cc X$).
Define the $*$-operation on $S_\infty^n(\cc X^\flat)$ by taking, for each $x^\flat \in S_\infty^n(\cc X^\flat)$, $x^{\flat*}$ to be the unique element of $S_\infty^n(\cc X^\flat)$ so that the equality $\mp{x^{\flat*}}{x} := \mp{x^\flat}{x^*}^*$ holds for any $x \in S_\infty^m(\cc X)$.
Equivalently, $\langle u^{\flat*}, u \rangle = \overline{\langle u^{\flat}, u^* \rangle}$ for any $u^\flat \in \cc X^\flat, u \in \cc X$, and $(a \otimes u^\flat)^* = a^* \otimes u^{\flat*}$.

 Similarly, define the positive wedge: a self-adjoint $x^\flat$ belongs to $S_\infty^n(\cc X^\flat)^+$ iff $\mp{x}{x^\flat} \in S_\infty^{mn+}$ for any $x \in S_\infty^m(\cc X)^+$. By \cite{effros1997matrix}, it suffices to verify this condition for $m=n$. Clearly all conditions of \Cref{d:ordered properties} are satisfied.

\begin{rmk}\label{rem:when cone dual}
\Cref{lem:when cone} shows that $S_\infty^n(\cc X^\flat)^+$ is a cone for every $n$ iff $\cc X^{\flat+}$ is a cone.
For different approach to proving this, recall \cite[Theorem 2.13]{Aliprantis-Tourky} that $\cc X^{\flat+}$ is a cone iff $\cc X^+ - \cc X^+$ is weakly dense in $\cc X_h$. Assuming this density, one can show that $S_\infty^n(\cc X)^+ - S_\infty^n(\cc X)^+$ is also weakly dense in $S_\infty^n(\cc X)_h$.
Using \Cref{s:order Banach}, we conclude that, if $\cc X_h$ is matricial generating, then $S_\infty^n(\cc X^\flat)^+$ is a cone, for every $n$.
\end{rmk}

\begin{rmk}\label{different positivity}
\cite[Lemma 2.3.8]{Schreiner-DISS} (see also \cite{Itoh}) gives several alternative descriptions of positivity in $S_\infty^n(\cc X^\flat)$. In particular, if $\cc X$ is a matricial order space, then $x^\flat = (x_{ij}^\flat)_{i,j=1}^n \in S_\infty^n(\cc X^\flat)$ is positive if an only if the corresponding linear map $S_\infty^n(\cc X) \to \C : (x_{ij})_{i,j=1}^n \mapsto \sum_{i,j} \langle x_{ij}^\flat, x_{ij} \rangle$ is positive.
\end{rmk}

If, in the above situation, $(\cc X, \alpha_n)$ and $(\cc X^\flat, \alpha_n^\flat)$ are matricial order operator spaces, then we shall assume that the matricial norms on the two spaces are consistent with the duality -- that is,
\begin{enumerate}
\item For any $x \in S_\infty^n(\cc X)$, $\alpha_n(x) = \sup \big\{ \| \mp{x^\flat}{x} \| : \alpha_n^\flat(x^\flat) \leq 1 \big\}$.
\item For any $x^\flat \in S_\infty^n(\cc X^\flat)$, $\alpha_n^\flat(x) = \sup \big\{ \| \mp{x^\flat}{x} \| : \alpha_n(x) \leq 1 \big\}$.
\end{enumerate}
This is the case, for instance, when $(\cc X^\flat, \alpha_n^\flat)$ is the canonical operator space dual, or pre-dual, of $(\cc X, \alpha_n)$.

\begin{prop}
If $\cc X$ is a matricial order operator space, then the canonical embedding into $\cc X \to \cc X^{\flat\flat}$ is completely isometric and completely bipositive.
\end{prop}

\begin{proof}[Sketch of a proof]
The canonical embedding is known to be completely isometric, see e.g.~\cite[Section 3]{effros-ruan-book}. The complete bipositivity follows from \Cref{different positivity}.
\end{proof}

For future use, we need a version of polarity. Suppose $K = (K_n)$ is a \emph{self-adjoint} (or \emph{hermitian}) \emph{matricial convex set} on $\cc X$ -- that is, $(K_n)$ is matricial convex in the sense of \Cref{matricial duality}, and $K_n \subset S_\infty^n(\cc X)_h$, for any $n$. Define the \emph{self-adjoint matricial polar} of $K$, denoted by $K^o = (K_n^o)$: 
for any $n$, $K_n^o = K_n^\pi \cap S_\infty^n(\cc X^\flat)_h$, or in other words, $K_n^o$ consists of those $x^\flat \in S_\infty^n(\cc X^\flat)_h$ for which $\mp{x^\flat}{x} \leq I_{mn}$ holds for any $x \in K_m$.

The results of \Cref{matricial duality} then yield:

\begin{prop}\label{hermitian dimension}
If $K = (K_n)$ is hermitian matricial convex on $\cc X$, then $x^\flat \in S_\infty^n(\cc X^\flat)_h$ belongs to $K_n^o$ iff $\mp{x^\flat}{x} \leq I_{n^2}$ for any $x \in K_n$.
\end{prop}

\begin{thm}\label{hermitian separation}
    Suppose $K = (K_n)$ is hermitian $\sigma(\cc X, \cc X^\flat)$-closed matricial convex subset of $\cc X$, so that $0 \in K_1$, and $x \in S_\infty^n(\cc X)_h$ doesn't belong to $K_n$. Then there exists $x^\flat \in K_n^o$ so that $\mp{x^\flat}{x} \not\leq I_{n^2}$.
\end{thm}

Next we establish a self-adjoint matricial version of the polarity between intersections and convex hulls. Consider $A = (A_n)$, where, for each $n$, $A_n \subset S_\infty^n(\cc Y)$ ($\cc Y$ is a topological space).
Denote by $\ncconvexhull{A}$ the smallest matricial convex set on $\cc Y$ containing $A$ (the \emph{noncommutative}, or \emph{matricial}, \emph{convex hull} of $A$), and by $\clncconvexhull{A}$ its ``level-wise'' weak (that is, $\sigma(\cc Y, \cc Y^\flat)$) closure: if $\ncconvexhull{A} = (B_n)$, then $\clncconvexhull{A} = (\overline{B_n})$. It is easy to see that $\clncconvexhull{A}$ is matricial convex as well.

For matricial sets $A = (A_n)$ and $B = (B_n)$, we shall write $A \cap B$ and $A \cup B$ for $(A_n \cap B_n)_n$ and $(A_n \cup B_n)_n$, respectively.

\begin{prop}\label{intersection vs hull}
Suppose $A = (A_n)$ and $B = (B_n)$ are self-adjoint matricial sets on a topological space $\cc X$, with $0 \in A_1 \cap B_1$. Then $(A \cap B)^o = \clncconvexhull{A^o \cup B^o}$.
\end{prop}

The above proposition can be applied when $\cc X$ is an ordered normed space, and $\cc X^\flat$ is the space of linear functionals, with its weak$^*$ topology.

\begin{proof}
The inclusion $(A \cap B)^o \supset \clncconvexhull{A^o \cup B^o}$ is straightforward.
To establish the converse, consider $x^\flat \in S_\infty^n(\cc X^\flat)_h$ which does not belong to $\clncconvexhull{A^o \cup B^o}$.
By \Cref{hermitian separation}, there exists $x \in S_\infty^n(\cc X)_h$ so that $\mp{x^\flat}{x} \not\leq I_{n^2}$, while $\mp{y^\flat}{x} \leq I_{mn}$ whenever $y^\flat \in A_m^o \cup B_m^o$.
The latter condition implies, by the non-commutative bipolar theorem \cite[Corollary 5.5]{effros1997matrix}, that $x \in A_n \cap B_n$, which shows that $x^\flat \not\in (A_n \cap B_n)^\circ$. Thus, $(A \cap B)^\circ \subset \clncconvexhull{A^\circ \cup B^\circ}$, finishing the proof. 
\end{proof}

Below we relate normality and generation properties of $\cc X$ with those of $\cc X^\flat$.

\begin{prop}\label{gen to nor}
    If a matricial order operator space $\cc X$ ($\cc X^\flat$) is $C$-matricial generating, then $\cc X^\flat$ $($resp.~$\cc X)$ is $C$-matricial normal.
\end{prop}

For future use, we record:

\begin{lem}\label{products}
    Suppose $u^\flat, u_1^\flat, u_2^\flat \in S_\infty^n(\cc X^\flat)$ and $u, u_1, u_2 \in S_\infty^m(\cc X)$ are such that $\begin{pmatrix} u_1^\flat & u^\flat \\ (u^\flat)^* & u_2^\flat \end{pmatrix} \in S_\infty^{2n}(\cc X^\flat)^+$, and $\begin{pmatrix}  u_1 & u  \\  u^* & u_2 \end{pmatrix} \in S_\infty^{2m}(\cc X)^+$. Then $\begin{pmatrix}    \mp{u_1}{u_1^\flat}  &   \mp{u}{u^\flat}  \\  \mp{u}{u^\flat}^*  &  \mp{u_2}{u_2^\flat}  \end{pmatrix}  \in \big(S_\infty^{2nm}\big)^+$.
\end{lem}

\begin{proof}
We clearly have
\[
\mp{\begin{pmatrix}  u_1 & u  \\  u^* & u_2 \end{pmatrix}}{\begin{pmatrix} u_1^\flat & u^\flat \\ (u^\flat)^* & u_2^\flat \end{pmatrix}} \in \big( S_\infty^{4nm} \big)^+ .
\]
The left side can be viewed as a $4 \times 4$ $S_\infty^{mn}$-valued matrix. The desired conclusion follows by restricting to rows and columns $1$ and $4$.
\end{proof}

As a by-product of the above reasoning, we obtain: 

\begin{lem}
    Suppose $u$ and $u^\flat$ are self-adjoint elements of $S_\infty^n(\cc X)$ and $S_\infty^m(\cc X^\flat)$, respectively. Suppose, furthermore, that $-v \leq u \leq v$ and $-v^\flat \leq u^\flat \leq v^\flat$, for some $v \in S_\infty^n(\cc X), v^\flat \in S_\infty^n(\cc X^\flat)$. Then $\| \mp{u^\flat}{u} \|^2 \leq \| \mp{v^\flat}{v} \|^2$.
\end{lem}

\begin{proof}[Proof of \Cref{gen to nor}]
We shall show that the matricial generation of $\cc X$ implies the matricial normality of $\cc X^\flat$; jumping from $\cc X^\flat$ to $\cc X$ is done similarly.
Suppose, for the sake of contradiction, that $\cc X$ is $C$-matricial generating, while $\cc X^\flat$ is not $C$-normal on the level $n$ -- that is, we can find $u^\flat, u_1^\flat, u_2^\flat \in S_\infty^n(\cc X^\flat)$ so that $\alpha_n^\flat(u^\flat) > C$, while $1 > \alpha_n^\flat(u_1^\flat) \vee \alpha_n^\flat(u_2^\flat)$, and $\begin{pmatrix} u_1^\flat & u^\flat \\ (u^\flat)^* & u_2^\flat \end{pmatrix} \geq 0$.
Find $u \in S_\infty^n(\cc X)$ so that $\alpha_n(u) < 1$, while $\|\mp{u}{u^\flat}\|_{n^2} > C$.
As $\cc X$ is $C$-matricial generating, we can find $u_1, u_2 \in S_\infty^n(\cc X)$ so that $\alpha_n(u_1), \alpha_n(u_2) < C$, yet $\begin{pmatrix}  u_1 & u  \\  u^* & u_2 \end{pmatrix} \geq 0$. By \Cref{products},
\[
\begin{pmatrix}    \mp{u_1}{u_1^\flat}  &   \mp{u}{u^\flat}  \\  \mp{u}{u^\flat}^*  &  \mp{u_2}{u_2^\flat}  \end{pmatrix}  \geq 0 .
\]
This, however, is impossible, since $\|\mp{u}{u^\flat}\| > C > \|\mp{u_1}{u_1^\flat}\| \vee \|\mp{u_2}{u_2^\flat}\|$.
\end{proof}

In the direction opposite to \Cref{gen to nor}, we establish:

\begin{prop}\label{nor to gen}
Suppose $\cc X$ is a matricial order operator space.
\begin{enumerate}
    \item 
If $\cc X$ is $C$-matricial normal, then $\cc X^\flat$ is $C$-matricial generating.
\item 
If $\cc X^\flat$ is $C$-matricial normal, and $\cc X$ is complete, then $\cc X$ is $C$-matricial generating.
\end{enumerate}
\end{prop}

\begin{proof}
We first establish (1), and then briefly describe the changes one needs to introduce into the proof to handle (2).

(1) In light of \Cref{reformulate normality}, we are reduced to showing: if $C \alpha_m(a+b) \geq \alpha_m(a-b)$ for any $a,b \in S_\infty^m(\cc X)^+$, then for any $x^\flat \in \ball(S_\infty^n(\cc X^\flat))_h$ there exists $y^\flat \in S_\infty^n(\cc X^\flat)^+$ so that $y^\flat \geq x^\flat \geq - y^\flat$, and $\alpha_n(y^\flat) \leq C$.

Equip the operator space ${\mathcal{X}}^2 := {\mathcal{X}} \oplus_\infty {\mathcal{X}}$ with componentwise order. For each $n$ consider $U_n = \{(a,b) \in S_\infty^n({\mathcal{X}}^2)^+ : \alpha_n(a+b) \leq 1/C\}$. One can check that $U = (U_n)$ is matricial convex.

Define $T : {\mathcal{X}}^2 \to {\mathcal{X}} : (a,b) \mapsto a-b$. Our assumption states that, for any $n$, $T_n(U_n) \subset \ball(S_\infty^n(\mathcal X))_h$ (we use shorthand $T_n = I_{S_\infty^n} \otimes T$). For the self-adjoint matricial polar (defined and described earlier in this section), we obtain: $(T_n(U_n))^o \supset \ball(S_\infty^n(\mathcal X^\flat))_h$.
It remains to show that for any $x^\flat \in (T_n(U_n))^o$ there exists $y \in C \ball(S_\infty^n({\cc X}^\flat))_h$ with $y^\flat \geq x^\flat \geq -y^\flat$.

Note that $x^\flat$ belongs to $(T_n(U_n))^o$ if, and only if, for any $x \in U_k$ we have $\mp{x^\flat}{T_k x} = \mp{T_n^\flat x^\flat}{x} \leq I_{kn}$ (here $T_n^\flat = T^\flat \otimes I_{S_\infty^n})$), that is, $T_n^\flat x^\flat \in U_n^o$.
It is easy to verify that $T_n^\flat x^\flat = (x^\flat, -x^\flat)$. Thus, $x^\flat \in (T_n(U_n))^o$ iff $(x^\flat,-x^\flat) \in U_n^o$.

Further, $U = V \cap ({\mathcal{X}}^2)^+$, where $({\mathcal{X}}^2)^+$ is a shorthand for $(S_\infty^k({\mathcal{X}}^2)^+)_k$, and $V = (V_k)$, where $V_k$ is the set of all $(a,b) \in S_\infty^k({\mathcal{X}}^2_h)$ with $\alpha_k(a+b) \leq 1/C$. 

The set $V = (V_k)$ is matricial convex; $(y^\flat, z^\flat) \in V_n^o$ iff $\mp{(y^\flat,z^\flat)}{(a,b)} = \mp{y^\flat}{a} + \mp{z^\flat}{b} \leq I_{kn}$ whenever $(a,b) \in V_k$. Note that $(a,-a) \in V_k$ for an arbitrary $a \in S_\infty^k(\mathcal X)_h$, hence $\mp{y^\flat - z^\flat}{a} \leq I_{kn}$ for any such $a$; this is possible only if $y^\flat = z^\flat$.
In this situation, $\mp{y^\flat}{a+b} \leq I_{kn}$ whenever $\alpha_k(a+b) \leq 1/C$, which happens iff $\alpha_n^\flat(y^\flat) \leq C$. To summarize, $V_n^o$ is the set of all pairs $(y^\flat,y^\flat)$, with $\alpha_n^\flat(y^\flat) \leq C$.
 
By \Cref{intersection vs hull}, 
$$
U^o = \big(V \cap ({\mathcal{X}}^2)^+ \big)^o = \clncconvexhull{V^o \cup - (({\mathcal{X}}^2)^\flat)^+} = \overline{ V^o - (({\mathcal{X}}^2)^\flat)^+ } 
$$
(as before, $\clncconvexhull{\cdot}$ denotes the level-wise $\sigma(\cc X^\flat, \cc X)$ closure).
Consequently, $(x^\flat,-x^\flat) \in U_n^o$ iff there exist nets $y_i^\flat \in C \ball(S_\infty^n(\cc X^\flat)_h)$ and $u_i^\flat, v_i^\flat \in \cc X^{\flat+}$ so that $(y_i^\flat-u_i^\flat)_i$ and $(y_i^\flat-v_i^\flat)_i$ converge to $x^\flat$ and $-x^\flat$, respectively, in $\sigma(\cc X^\flat, \cc X)$.
A compactness argument shows that, by passing to a subnet, we can assume that $(y_i^\flat)$ converges to $y^\flat \in C \ball(S_\infty^n(\cc X^\flat)_h)$, in $\sigma(\cc X^\flat, \cc X)$.
For any $x \in S_\infty(\cc X)^+$,
$$
\liminf_i \mp{y_i^\flat - x^\flat}{x} = \liminf_i \mp{u_i^\flat}{x} \geq 0 ,
$$
hence $y^\flat \geq x^\flat$. Likewise $y^\flat \geq -x^\flat$.

(2) Next we outline the changes needed to prove that, if $\cc X^\flat$ is matricially normal, then $\cc X$ is matricially generating. We need to establish: if $C \alpha_m^\flat(a^\flat+b^\flat) \geq \alpha_m^\flat(a^\flat-b^\flat)$ for any $a^\flat,b^\flat \in S_\infty^m(\cc X^\flat)^+$, then for any $\vr \in (0,1/3)$, and $x \in \ball(S_\infty^n(\cc X)_h)$ there exists $y \in S_\infty^n(\cc X^\flat)^+$ so that $y \geq x \geq - y$, and $\alpha_n(y) \leq C(1 + \vr)$.

Define the matricial sets $U$, $V$, etc., interchanging the roles of $\cc X$ and $\cc X^\flat$. For instance, $U_n = \{(a^\flat,b^\flat) \in S_\infty^n(({\mathcal{X}}^\flat)^2)^+ : \alpha_n^\flat(a^\flat+b^\flat) \leq 1/C\}$.
We arrive at $U^o = \overline{ V^o - ({\mathcal{X}}^2)^+ }$, where, in the right side, we take the closure in the norm (equivalently, $\sigma(\cc X, \cc X^\flat)$) topology.

For a moment, consider $a \in \ball(S_\infty^n(\cc X)_h)$. Then $(a,-a) \in U_n^o$, hence, for every $\delta > 0$, there exist $z \in C \ball(S_\infty^n(\cc X)_h)$ and $u, v \in S_\infty^n(\cc X)^+$ so that $\alpha_n((z-u)-a), \alpha_n((z-v)+a) < \delta$. 
Let $z' = (u+v)/2$, and $a' = (u-v)/2$, then $-z' \leq a' < z'$. By the triangle inequality, $\alpha_n(z' - z) < \delta$, hence $\alpha_n(z') < 1 + \delta$, and
$$ \alpha_n(a' - a) = \alpha_n \Big( \big( \frac{u+v}2 - z\big) + \big(z - u - a\big) \Big) < 2 \delta . $$ 

Now take $x = x_0 \in \ball(S_\infty^n(\cc X)_h)$ and $\vr \in (0,1)$. Use the preceding paragraph to write $x_0 = x_1 + x_0'$ so that all summands are self-adjoint, $\alpha_n(x_1) < \vr/3$, and there exists a positive $z_0$ with $-z_0 \leq x_0' \leq z_0$, and $\alpha_n(z_0) < (1 + \vr/3)C$.
Further, write $x_1 = x_2 + x_1'$ so that all summands are self-adjoint, $\alpha_n(x_2) < (\vr/3)^2$, and there exists a positive $z_1$ with $-z_1 \leq x_1' \leq z_1$, and $\alpha_n(z_1) < 3^{-1} \vr( 1 + \vr/3) C$.
Proceeding further in the same manner, we write, for each $k$, $x = x_0' + x_1' + \ldots + x_{k-1}' + x_k$, where $\alpha_n(x_k) < (\vr/3)^k$, and for any $j$, there exists $y_j \geq \pm x_j'$, with $\alpha_n(z_j) < (\vr/3)^j (1+\vr/3)C$. Let $y = \sum_{j=0}^\infty z_j$, then 
$$
\alpha_n(y) < \sum_{j=0}^\infty \Big(\frac\vr3\Big)^j \Big(1+\frac\vr3\Big)C = \frac{1+\vr/3}{1-\vr/3} C < (1+\vr) C .
$$
Also, $x = \sum_{j=0}^\infty x_j'$, hence $y - x = \sum_{j=0} (z_j - x_j') \geq 0$, and similarly, $y + x \geq 0$. Thus, $y$ has the desired properties.
\end{proof}

\begin{rmk}
\cite{Schreiner-JFA} used Haagerup tensor products to prove that, if $\cc X$ is $1$-normal and $1$-generating, then so is $\cc X^\flat$.
\end{rmk}

\section{Minimal and maximal matricial order spaces}\label{min max structures}

It is well-known (see e.g. \cite{Blecher1991tensor}, \cite[Section 3.3]{effros-ruan-book}, or \cite[Section 3]{Pisier-Op-Spaces}) that a Banach space can be equipped with different operator space structures, among them the maximal and minimal ones.
In this section, we endow a given ordered Banach space with operator space \emph{and} matricial order structures -- specifically, the maximal and minimal ones -- and study these structures.

Throughout, a Banach space $\cc X$ is assumed to be complex, with its positive wedge $\cc X^+$, and an antilinear involution $x \mapsto x^*$, which leaves $\cc X^+$ invariant.

For future use, we need to observe that self-adjoint elements of $\cc X$ can be characterized using positive functionals.

\begin{prop}\label{rich in +}
Suppose $\cc X^+$ is the positive cone of an ordered Banach space $\cc X$. Then $x \in \cc X$ is self-adjoint if and only if $\langle x^\flat, x \rangle \in \R$ for any $x^\flat \in \cc X^{\flat+}$.
\end{prop}

\begin{proof}
If $x \in \cc X_h$, then clearly $\langle x^\flat, x \rangle \in \R$ for any $x^\flat \in \cc X^{\flat+}$. For the opposite implication, note that, by \cite[Theorem 2.13]{Aliprantis-Tourky}, $\cc X^{\flat+} - \cc X^{\flat+}$ is weak$^*$ dense in $\cc X^\flat_h$. Thus, if $\langle x^\flat, x \rangle \in \R$ for any $x^\flat \in \cc X^{\flat+}$, then the same holds for any $x^\flat \in \cc X_h$, which shows the self-adjointness of $x$.
\end{proof}

Begin with describing the minimal structure on  an ordered Banach space $\cc X$.

\begin{defn}\label{min-structure}
For an ordered Banach space $\cc X$, equip $S_\infty^n(\cc X)$ with its injective Banach space tensor product (that is, $\cc X$ has the minimal operator space structure). 
That is, for $x \in S_\infty^n(\cc X)$, let $\|x\| = \sup_{a \in \ball(S_1^n)} \|{\mathrm{tr}} \otimes I_{\cc X}(a \cdot x)\|_{\cc X}$ (here ${\mathrm{tr}} \otimes I_{\cc X}(a \cdot (b \otimes y)) = \tra(ab) y$). 
An $ x \in S_\infty^n(\MIN(\cc X))$ belongs to the positive wedge if ${\mathrm{tr}} \otimes I_{\cc X}(a \cdot x) \in \cc X^+$ whenever $a \in S_1^{n+}$. 
\end{defn}
By an extreme point argument, $\|x\| = \sup_{\xi,\eta \in \ball(\ell_2^n)} \| \langle \xi | x | \eta \rangle \|_{\cc X}$.
Similarly, a self-adjoint $x$ is positive iff $\langle \xi | x | \xi \rangle \in \cc X^+$ for any $\xi \in \ell_2^n$. By \Cref{rich in +}, the self-adjointness is automatic if $\cc X^+$ is a cone (and not merely a wedge).

Clearly, the object described here is indeed a matricial convex wedge; 
$S_\infty^1(\MIN(\cc X))^+$ is identified with $\cc X^+$. By \Cref{lem:when cone}, $S_\infty^n(\MIN(\cc X))^+$ is a cone whenever $\cc X^+$ is.

Any $x = \sum_{i,j} E_{ij} \otimes x_{ij} \in S_\infty^n(\MIN(\cc X))^+$ is self-adjoint. Indeed, the positivity of $\langle \xi | x | \xi \rangle$ with $|\xi\rangle = |i\rangle$ tells us that $x_{ii} = x_{ii}^* \geq 0$. Further, taking  $|\xi\rangle = |i\rangle + |j\rangle$, we obtain $x_{ij} = x_{ji}^*$.

In yet another reformulation, $x \in S_\infty^n(\cc X)^+$ iff $\langle a \otimes x^\flat, x \rangle \geq 0$ whenever $a \in S_\infty^{n+}, x^\flat \in \cc X^{\flat+}$ (see \cite[Theorem 2.13]{Aliprantis-Tourky}).
By the self-adjointness noted earlier, the positive cone in $S_\infty^n(\cc X)$ is nothing but the ``injective cone'' $C_i \subset (S_\infty^n)_h \otimes \cc X_h$ from \cite{Wittstock}.
In other words, $C_i$ consists of those $x \in S_\infty^n(\cc X)$ which define a positive operator $\oper{x}$ from $S_1^n$ to $\cc X$. 
Then $\|x\|_{S_\infty^n(\MIN(\cc X))} = \|\oper{x}\|$. 

Some simple properties of the $\MIN$ structure are noted below. 

\begin{prop}\label{min passes to subspace}
$(1)$ Suppose $\cc X$ is an ordered complex Banach space, and $\cc Y$ is its involution-invariant subspace. Then $\MIN(\cc Y)$ embeds into $\MIN(\cc X)$ as a matricial order operator space -- that is, it inherits the matricial norms and the matricial positive wedges from the latter.

$(2)$ If $\cc Y$ is an ordered Banach space, $\cc X$ is a matricial order operator space, and $T : \cc X \to \cc Y$ is positive, then it is completely positive as a map from $\cc X$ to $\MIN(\cc Y)$.
\end{prop}

\begin{proof}[Sketch of a proof]
Recall that $y \in S_\infty^n(\MIN(\cc Y))^+$ iff $\langle \xi |y| \xi \rangle \in \cc Y^+$ for any $\xi \in \ell_2^n$. From this, (1) immediately follows. For (2), note that, if $x \in S_\infty^n(\cc X)^+$, then $\langle \xi |x| \xi \rangle \in \cc X^+$ for any $\xi \in \ell_2^n$.
\end{proof}

Next we show that the matricial order space $\MIN(\cc X)$ has certain desirable properties for a class of ordered spaces $\cc X$.

\begin{defn}
We say that an ordered Banach space $\cc X$ is \emph{$\MIN$-nice} if $\|x\| \leq \|x_1\| \vee \|x_2\|$ whenever 
\begin{equation}
x_1, x_2 \geq 0 , \, \, {\textrm{  and  }} \, \, 
t^2 x_1 + s^2 x_2 \geq 2 ts \Re(\omega x) 
\, \, {\textrm{  for any   }} \, \,  t,s \in \R , \, |\omega| = 1 .
\label{eq:min-nice}
\end{equation}
\end{defn}

As an example of an ordered space with desired properties, consider $\cc X = S_\infty^n$, with a unitarily invariant norm $\| \cdot \|$ and natural order. The condition \eqref{eq:min-nice} is equivalent to $\begin{pmatrix} x_1 & x \\ x^* & x_2 \end{pmatrix} \geq 0$.
In this case $\|x\|^2 \leq \|x_1\| \|x_2\|$ (see \cite{horn-mathias} for this, and more results in the same vein). This gives us one example of $\MIN$-niceness. In \Cref{lattice} below, we shall see other instances of this, coming from Banach lattices.

\begin{rmk}
    If $\cc X$ is $\MIN$-nice, then $\cc X_h$ is $1$-normal, as a real ordered Banach space. Indeed, suppose $x, y \in \cc X_h$ satisfy $\pm x \leq y$. Then $t^2 y + s^2 y \geq 2 ts \Re(\omega x)$ for $t,s,\omega$ as in \eqref{eq:min-nice}, hence $\|x\| \leq \|y\|$, per definition.
\end{rmk}

\begin{prop}\label{min structure on lattices}
An ordered Banach space $\cc X$ is $\MIN$-nice if and only if the matricial order space $\MIN(\cc X)$ described above is $1$-matricial normal.
\end{prop}

\begin{proof}
Suppose first $\cc X$ is $\MIN$-nice. Consider $\begin{pmatrix}  x_1 & x  \\  x^* & x_2  \end{pmatrix} \in S_\infty^{2n}(\cc X)^+$, and $\|x\| > 1$; show that $\|x_1\| \vee \|x_2\| > 1$.

Find $\xi_1, \xi_2 \in \ball(\ell_2^n)$ so that $\|u\| > 1$, where $u = \langle \xi_1 | x | \xi_2 \rangle$. Let $u_i = \langle \xi_i | x_i | \xi_i \rangle$. Testing the positivity of $\begin{pmatrix}  x_1 & x  \\  x^* & x_2  \end{pmatrix} \in S_\infty^{2n}(\cc X)^+$ on $t \xi_1 \oplus s \xi_2$, we see that the inequality $|t|^2 u_1 + |s|^2 u_2 + 2 \Re (ts u) \geq 0$ holds for any $t,s \in \C$. So, $u_1, u_2$ satisfy \eqref{eq:min-nice}; by $\MIN$-niceness, $\|u_1\| \vee \|u_2\| > 1$, hence $\|x_1\| \vee \|x_2\| > 1$.

Conversely, suppose $\MIN(\cc X)$ described above is $1$-matricial normal, and $x, x_1, x_2 \in \cc X$ satisfy \eqref{eq:min-nice}. Then $\begin{pmatrix}  x_1 & x  \\  x^* & x_2  \end{pmatrix} \in S_\infty^{2}(\cc X)^+$, hence by normality, $\|x\| \leq \|x_1\| \vee \|x_2\|$.
\end{proof}

In \Cref{lattice} we shall see examples of minimal matricial order spaces which are not matricial generating.

Next we describe the maximal structure on an ordered space $\cc X$ (denoted by $\MAX(\cc X)$).

\begin{defn}\label{max structure}
The sequence of matricial norms on $\MAX(\cc X)$ is defined classically: for $x \in S_\infty^n(\cc X)$, $\norm{x} = \sup \norm{I \otimes T(x)}$, with the supremum taken over all contractions $T : \cc X \to S_\infty^m$ (so $I \otimes T(x) \in S_\infty^{mn}$).
Define the positive wedge as follows: $x \in S_\infty^n(\MAX(\cc X))^+$ if $x = x^*$, and $I \otimes T(x) \in S_\infty^{mn+}$ for any positive $T$ as above; one can verify that this is indeed a (closed) matricial wedge, and $S_\infty^1(\MAX(\cc X))^+ = \cc X^+$. 
\end{defn}

Note that $S_\infty^n(\MAX(\cc X))^+ \subset S_\infty^n(\MIN(\cc X))^+$. Indeed, suppose $x$ lies outside of $S_\infty^n(\MIN(\cc X))^+$. Then for some $\xi \in \ell_2^n$, $\langle \xi |x| \xi \rangle \notin \cc X^+$. Find $x^\flat \in \cc X^{\flat+}$ so that $\langle x^\flat, \langle \xi |x| \xi \rangle \rangle \notin \R^+$.
Equivalently, $\langle \xi| I_{S_\infty^n} \otimes x^\flat (x) |\xi \rangle < 0$. Interpret $x^\flat$ as a positive operator $\cc X \to \C$, and observe that $I_{S_\infty^n} \otimes x^\flat (x) \notin S_\infty^{n+}$, which shows that $x \notin S_\infty^n(\MAX(\cc X))^+$. In particular, if $\cc X^+$ is a cone, then so are the matricial wedges in $\MAX(\cc X)$.

\begin{rmk}
For $x \in S_\infty^n(\cc X)_h$, the following statements are equivalent:
\begin{enumerate}
\item $x \in S_\infty^n(\MAX(\cc X))^+$.
\item $I \otimes T(x) \in S_\infty^m(B(H))^+$ for any Hilbert space $H$, and any positive $T : \cc X \to B(H)$.
\item $I \otimes T(x) \in S_\infty^m(\cc Y)^+$ for any positive $T : \cc X \to \cc Y$, where $\cc Y$ is a matricial order operator space, and $\cc Y^+$ is a cone.
\end{enumerate}
Indeed, the implications $(3) \Rightarrow (1) \Leftrightarrow (2)$ are straightforward. From \cite{werner2002matrix}, any matricial order operator space $\cc Y$ embeds completely bipositively into $B(H)$, whence $(2) \Rightarrow (3)$. 
\end{rmk}

The term ``maximal'' is justified here:

\begin{prop}
If $\cc X$ is an ordered Banach space, $\cc Y$ is a matricial ordered operator space for which $\cc Y^+$ is a cone, and $T : \cc X \to \cc Y$ is positive, then it determines a completely positive map $\MAX(\cc X) \to \cc Y$.
\end{prop}

\begin{proof}[Sketch of a proof]
By \cite{werner2002matrix}, $\cc Y$ embeds into $B(H)$ completely bipositively.
\end{proof}

We can also provide an alternative description of $S_\infty^n(\MAX(\cc X))^+$.

\begin{lem}\label{cone in max}
For every $n$, $S_\infty^n(\MAX(\cc X))^+$ coincides with the norm $($equivalently, weak$)$ closure of finite sums $\sum_i a_i \otimes x_i$, with $a_i \in S_\infty^{n+}, x_i \in \cc X^+$.
\end{lem}

\begin{proof}
Clearly the elementary tensors $a_i \otimes x_i$ (with positive $a_i, x_i$) belong to the positive cone of $S_\infty^n(\MAX(\cc X))$, which gives us inclusion in one direction.

To prove the opposite, denote the closure of finite sums $\sum_i a_i \otimes x_i$ ($a_i \in S_\infty^{n+}, x_i \in \cc X^+$) by $K_n$. Pick a self-adjoint $x \in S_\infty^n(\MAX(\cc X)) \backslash K_n$, and show that $x \notin S_\infty^n(\MAX(\cc X))^+$.
The family $(K_n)$ is matricial convex, hence, by \cite[Theorem 5.4]{effros1997matrix} (see also \Cref{ss:duality}), there exists a self-adjoint $x^\flat \in S_\infty^n \otimes \cc X^\flat$ so that $\mp{x^\flat}{y} \leq I_{S_\infty^{mn}}$ for any $m$, and any $y \in K_m$, while $\mp{x^\flat}{x} \not\leq I_{S_\infty^{mn}}$.
As $(K_m)$ is a matricial cone, the first condition means that $-\mp{x^\flat}{y} \in S_\infty^{nm+}$ for any $y \in K_m$. Specializing to $m=1$, we conclude that the operator $T$, determined by $-x^\flat$, is positive. On the other hand, $I_{S_\infty^n} \otimes T(x) = - \mp{x^\flat}{x} \notin S_\infty^{n^2+}$ (otherwise we would have $\mp{x^\flat}{x} \leq 0 \leq I_{S_\infty^{mn}}$, so $x \notin S_\infty^n(\MAX(\cc X))^+$.
\end{proof}

\begin{rmk}\label{max control dimension}
    The above proof also shows that a self-adjoint $x \in S_\infty^n \otimes \cc X$ is positive in $S_\infty^n(\MAX(\cc X))$ iff $I \otimes T \in S_\infty^{n^2+}$ for any positive $T : \cc X \to S_\infty^n$.
\end{rmk}

We next show that an ``appropriate'' quotient of a maximal matricial order space is again maximal (cf.~\Cref{min passes to subspace}(1)). To be able to define a ``natural'' order on the quotient space, we mod out by kernels, as described in the paragraph preceding \Cref{quotient-generating}. 
More specifically, suppose $\cc X$ is an ordered Banach space, with positive cone $\cc X^+$, and $J = \ker U$, where $U : \cc X \to \cc Y$ is a positive map, $\cc Y$ is an ordered space, with $\cc Y^+$ being a cone. Turn $\cc X/J$ into an ordered space by defining the involution in the obvious way, and letting $(\cc X/J)^+$ be the closure of $\cc X^+ + J$ (as shown in \cite[Section 3.2]{humeniuk2023extension}, this is a cone).

\begin{prop}\label{max quotients}
In the above notation, $\MAX(\cc X)/J = \MAX(\cc X/J)$ as matrical order operator spaces.
\end{prop}

\begin{proof}
By \cite[Proposition 3.3]{Pisier-Op-Spaces}, $\MAX(\cc X)/J = \MAX(\cc X/J)$ as operator spaces. 

For the order structures, show first that $S_\infty^n(\MAX(\cc X)/J)^+ \subset S_\infty^n(\MAX(\cc X/J))^+$.
Recall from the explanation preceding \Cref{quotient-generating} that $u \in S_\infty^n(\MAX(\cc X)/J)^+$ iff $u = \lim_m q_n(x_m)$, where $x_m \in S_\infty^n(\MAX(\cc X))^+$, and $q_n$ is the $n$-amplification of the quotient map $q : \cc X \to \cc X/J$.
By \Cref{cone in max}, we can assume that $x_m = \sum_k a_{mk} \otimes y_{mk}$, with $a_{mk} \in S_\infty^{n+}$ and $y_{mk} \in \cc X^+$. Thus, $u = \lim_m \sum_k a_{mk} \otimes q_n(y_{mk})$. As $\sum_k a_{mk} \otimes q_n(y_{mk})$ lies in $S_\infty^n(\MAX(\cc X/J))^+$, the same is true for $u$.

On the other hand, any $v \in S_\infty^n(\MAX(\cc X/J))^+$ can be written as $v = \lim_m v_m$, with $v_m = \sum_k b_{mk} \otimes w_{mk}$, where $b_{mk} \in S_\infty^{n+}$ and $w_{mk} \in (\cc X/J)^+$.
From our description of $(\cc X/J)^+$, we can assume that $w_{mk} = q(z_{mk})$, with $z_{mk} \in \cc X^+$. Then $v_m = q_n(\sum_k b_{mk} \otimes w_{mk})$, and $\sum_k b_{mk} \otimes w_{mk} \in S_\infty^n(\cc X)^+$, so $v_m \in S_\infty^n(\MAX(\cc X)/J)^+$, and therefore, $v \in S_\infty^n(\MAX(\cc X)/J)^+$.
\end{proof}

\begin{defn}\label{max-nice}
We say that an ordered Banach space $\cc X$ is \emph{$\MAX$-nice} if for every $x \in \cc X$ with $\|x\| < 1$ and every $\vr > 0$ there exist $x_1, \ldots, x_n \in \cc X^+$ and $\xi_1, \ldots, \xi_n, \eta_1, \ldots, \eta_n \in \C$ so that $\|\sum_i |\xi_i|^2 x_i\|, \|\sum_i |\eta_i|^2 x_i\| < 1$, and $\|x - \sum_i \xi_i \overline{\eta_i} x_i\| < \vr$.
\end{defn}

\begin{rmk}
     $\cc X_h$ being $1$-generating means that, for every $x \in \cc X_h$ with $\|x\| < 1$, there exists $y \in \cc X^+$ so that $\pm x \leq y$, and $\|y\| < 1$. Let $x_1 = (y+x)/2$, $x_2 = (y-x)/2$, $\xi_1 = \eta_1 = \xi_2 = 1$, and $\eta_2 = -1$. Then $x_1, x_2 \in \cc X^+$, $\sum_i |\xi_i|^2 x_i = \sum_i |\eta_i|^2 x_i = y$ (hence both have the norm less than $1$), and $\sum_i \xi_i \overline{\eta_i} x_i = x_1 - x_2 = x$ -- thus, the conditions of $\MAX$-niceness hold.
\end{rmk}

\begin{prop}
    Suppose $\cc X$ is an ordered Banach space. Then $\MAX(\cc X)$ is $1$-matricial generating iff $\cc X$ is $\MAX$-nice.
\end{prop}

\begin{proof}
If $\MAX(\cc X)$ is $1$-matricial generating, then for any $x \in \cc X$ with $\|x\| < 1$ there exists a positive $\begin{pmatrix} y_1 & x \\ x^* & y_2 \end{pmatrix}$ so that $\|y_1\|, \|y_2\| < 1$.
By \Cref{cone in max}, for any $\delta > 0$ there exist $x_1, \ldots, x_n \in \cc X^+$ and vectors $\begin{pmatrix} \xi_i \\ \eta_i \end{pmatrix} \in \C^2$ so that
$$
\norm{\begin{pmatrix} y_1 & x \\ x^* & y_2 \end{pmatrix} - \sum_i \begin{pmatrix} \xi_i \\ \eta_i \end{pmatrix} \big( \overline{\xi_i}, \overline{\eta_i} \big) \otimes x_i} < \delta . 
$$
One can check that, for $\delta$ small enough, $(x_i), (\xi_i), (\eta_i)$ are as in the definition of being $\MAX$-nice. Note that this direction does not require the completeness of $\cc X$.

Conversely, suppose $\cc X$ is $\MAX$-nice, and consider $x \in S_\infty^n(\cc X)$ with $\|x\| < 1$. By \cite[Theorem 3.1]{Pisier-Op-Spaces}, we can find a ``diagonal'' $z = \sum_{i=1}^N E_{ii} \otimes z_i \in S_\infty^N(\cc X)$ and $n \times N$ matrices $a, b \in S_\infty^{n,N}$ so that $\|a\|, \|b\| < 1$, $\max_i \|z_i\| < 1$, and $x = a \cdot y \cdot b^* = \sum_{i=1}^N a E_{ii} b^* \otimes z_i$.

By the definition of $\MAX$-niceness, for each $\delta > 0$ and $i$ there exist $\xi_{ik}, \eta_{ik} \in \C$ and $x_{ik} \in \cc X^+$ ($1 \leq k \leq K_i$) so that $\|\sum_k |\xi_{ik}|^2 x_{ik}\|, \|\sum_k |\eta_{ik}|^2 x_{ik}\| < 1$, and $\|z_i - \sum_k \xi_{ik} \overline{\eta_{ik}} x_{ik}\| < \delta$. 

Let
$$
u_1  = \sum_{i=1}^N E_{ii} \otimes \sum_k |\xi_{ik}|^2 x_{ik} ,  \, \, u_2  = \sum_{i=1}^N E_{ii} \otimes \sum_k |\eta_{ik}|^2 x_{ik} ,  \, \, u  = \sum_{i=1}^N E_{ii} \otimes \sum_k \xi_{ik} \overline{\eta_{ik}} x_{ik} ,
$$
then
$$
\widetilde{u} = \begin{pmatrix} u_1 & u \\ u^* & u_2 \end{pmatrix} = \sum_{i=1}^N E_{ii} \otimes \sum_k \begin{pmatrix} |\xi_{ik}|^2 & \xi_{ik} \overline{\eta_{ik}} \\ \overline{\xi_{ik}} \eta_{ik} & |\eta_{ik}|^2 \end{pmatrix} \otimes x_{ik} \in S_\infty^{2N}(\MAX(\cc X))^+ .`
$$
Note that $\|y - u\| = \max_i \|z_i - \sum_k \xi_{ik} \overline{\eta_{ik}} x_{ik}\| < \delta$, while $\|u_1\|, \|u_2\| < 1$. Consequently,
$$
\begin{pmatrix} v_1 & v \\ v^* & v_2 \end{pmatrix} := \begin{pmatrix} a & 0 \\ 0 & b \end{pmatrix} \cdot \widetilde{u} \cdot \begin{pmatrix} a & 0 \\ 0 & b \end{pmatrix}^* = \sum_{k,i} \begin{pmatrix} |\xi_{ik}|^2 a E_{ii} a^* & \xi_{ik} \overline{\eta_{ik}} a E_{ii} b^* \\ \overline{\xi_{ik}} \eta_{ik} b E_{ii} a^* & |\eta_{ik}|^2 b E_{ii} b^* \end{pmatrix} \otimes x_{ik}  \in S_\infty^{2n}(\MAX(\cc X))^+ ,
$$
with $\|x-v\| < \delta$, and $\|v_1\|, \|v_2\| < 1$.

Now fix $\vr \in (0,1/2)$ and $x \in \ball(S_\infty^n(\MAX(\cc X)))$. We need to find $x_1, x_2 \in S_\infty^n(\MAX(\cc X))$ so that $\begin{pmatrix} x_1 & x \\ x^* & x_2 \end{pmatrix} \in S_\infty^{2n}(\MAX(\cc X))^+$, and $\|x_1\|, \|x_2\| < 1+\vr$. As $\cc X$ is complete, we can rely on an iterative procedure for this purpose. 

Begin by finding (in a manner described above) $z_0, x_{01}, x_{02} \in S_\infty^n(\MAX(\cc X))$ so that $\begin{pmatrix} x_{01} & z_0 \\ x_{00}^* & x_{02} \end{pmatrix} \in S_\infty^{2n}(\MAX(\cc X))^+$, $\|x_{01}\|, \|x_{02}\| < 1$, and $\|y_0\| < \vr/2$, where $y_0 = x - z_0$.

Now find $z_1, x_{11}, x_{12} \in S_\infty^n(\MAX(\cc X))$ so that $\begin{pmatrix} x_{11} & z_1 \\ z_1^* & x_{12} \end{pmatrix} \in S_\infty^{2n}(\MAX(\cc X))^+$, $\|x_{11}\|, \|x_{12}\| < \vr/2$, and $\|y_1\| < \vr/2^2$, where $y_1 = y_0 - z_1$.
Proceeding further in the same manner, the $k$-th step gives us $z_k, x_{k1}, x_{k2} \in S_\infty^n(\MAX(\cc X))$ so that $\begin{pmatrix} x_{k1} & z_k \\ z_k^* & x_{k2} \end{pmatrix} \in S_\infty^{2n}(\MAX(\cc X))^+$, $\|x_{k1}\|, \|x_{k2}\| < \vr/2^k$, and $\|y_k\| < \vr/2^{k-1}$, where $y_k = y_{k-1} - z_k$.
Then $x = \sum_{k=0}^\infty z_k$, and one can check that $x_\ell = \sum_{k=0}^\infty x_{k\ell}$ ($\ell = 1,2$) have the desired properties.
\end{proof}

For dual spaces, we can characterize membership in the positive cone using adjoint operators. 

\begin{lem}\label{use predual}
    Suppose $\cc Y$ is an ordered Banach space, and $\cc Y^+$ is generating. Then a self-adjoint $x \in S_\infty^n(\MAX(\cc Y^\flat))^+$ if and only if $I \otimes U^\flat(x) \in S_\infty^{n^2+}$ for any positive $U : S_1^n \to \cc Y$. 
\end{lem}

\begin{proof}
    The ``only if'' part of the statement is straightforward, since $U$ is positive iff $U^\flat$ is.

    To establish the ``if'' part, denote for convenience $\cc X = \cc Y^\flat$. Suppose a self-adjoint $x \in S_\infty^n \otimes \cc X$ (which therefore lies in $(S_\infty^n)_h \otimes \cc X_h$) is not positive in $S_\infty^n(\MAX(\cc X))$. We need to find $U : S_1^n \to \cc Y$ so that $I \otimes U^\flat(x)$ is not positive.
    
    Per \Cref{max control dimension}, we can find a positive $T : \cc X \to S_\infty^n$ so that $I \otimes T(x)$ is not positive. That is, there exists $\xi = \sum_j \eta_j \otimes \zeta_j \in \ell_2^{n^2}$ ($\eta_j, \zeta_j \in \ell_2^n$) so that $\langle \xi | (I \otimes T) x | \xi \rangle < 0$.
    Write $x = \sum_i a_i \otimes x_i$ ($a_i \in (S_\infty^n)_h, x_i \in \cc X_h$), then
    $$
    0 > \sum_{i,j,k} \langle \eta_j | a_i | \eta_k \rangle \langle \zeta_j | T x_i | \zeta_k \rangle = \sum_{i,j,k} \langle \eta_j | a_i | \eta_k \rangle \big\langle  T x_i , |\zeta_k \rangle \langle \zeta_j | \big\rangle = \sum_{i,j,k} \langle \eta_j | a_i | \eta_k \rangle \big\langle T^\flat (|\zeta_k \rangle \langle \zeta_j |) , x_i\big\rangle .
    $$
Above, $\big\langle  T x_i , |\zeta_k \rangle \langle \zeta_j | \big\rangle$ represents the action of $T x_i \in S_\infty^n$ on the rank one operator $|\zeta_k \rangle \langle \zeta_j | \in S_1^n$; then $T^\flat (|\zeta_k \rangle \langle \zeta_j |)$ is an element of $\cc X^\flat$, and $\big\langle T^\flat (|\zeta_k \rangle \langle \zeta_j |) , x_i\big\rangle$ is its action on $x_i \in \cc X$.
By \Cref{p:into cone}, there exists a positive $U : S_1^n \to \cc Y$ ``approximating'' $T^\flat$, in the sense
$$\sum_{i,j,k} \langle \eta_j | a_i | \eta_k \rangle \big\langle x_i , U (|\zeta_k \rangle \langle \zeta_j |) \big\rangle < 0,$$ or equivalently, $\langle \xi | (I \otimes U^\flat) x | \xi \rangle < 0$. Thus, $U$ has the desired properties.
\end{proof}

We close this section by establishing ``min-max duality.''

\begin{prop}\label{max*=min}
If $\cc X$ is an ordered Banach space, then $\MAX(\cc X)^\flat = \MIN(\cc X^\flat)$ as matricial order operator spaces.
\end{prop}

\begin{proof}
It is well known that the matricial norms are in duality. It remains to establish the duality of positive cones -- that is, $x^\flat \in S_\infty^n(\MIN(\cc X^\flat))$ is positive iff $\mp{x^\flat}{x} \in S_\infty^{mn+}$ for any $x \in S_\infty^m(\MAX(\cc X))^+$.

Recall that $x^\flat \in S_\infty^n(\MIN(\cc X^\flat))^+$ iff the corresponding weak$^*$-continuous operator $T : \cc X^{\flat\flat} \to S_\infty^n$ is positive. By \Cref{p:Banach Alaoglu}, this is equivalent to the restriction of this operator to $\cc X$ (which we shall also denote by $T$) being positive.
If this happens, then, by the definition of the positive cone of $\MAX(\cc X)$, $I \otimes T(x) = \mp{x^\flat}{x} \in S_\infty^{mn+}$ for any $x \in S_\infty^m(\MAX(\cc X))^+$.

Conversely, suppose $I \otimes T(x) = \mp{x^\flat}{x} \in S_\infty^{mn+}$ for any $x \in S_\infty^m(\MAX(\cc X))^+$. Then in particular $Tx \in S_\infty^{n+}$ for any $x \in \cc X^+$, so $T$ is positive on $\cc X^+$. As noted above, this gives the positivity of $x^\flat$.
\end{proof}

\begin{prop}\label{mix*=max}
If $\cc X$ is an ordered Banach space, and $\cc X_h$ is generating, then $\MIN(\cc X)^\flat = \MAX(\cc X^\flat)$ as matricial order operator spaces.
\end{prop}

\begin{proof}
Again, the duality between matricial norms is well known. To show that the cones in question are dual to each other, recall that $x^\flat \in S_\infty^n(\MIN(\cc X)^\flat)^+$ iff for any positive weak$^*$-continuous $S : \cc X^\flat \to S_\infty^m$ (corresponding to a positive element of $S_\infty^m(\MIN(\cc X))$), we have $I \otimes S(x^\flat) \in S_\infty^{mn+}$.
On the other hand, by the definition of the positive cone in $\MAX$ spaces, $x^\flat \in S_\infty^n(\MAX(\cc X^\flat))^+$ iff for any positive $T : \cc X^\flat \to S_\infty^m$ we have $I \otimes T(x^\flat) \in S_\infty^{mn+}$. By \Cref{use predual}, we can take $T$ to be weak$^*$ continuous.
\end{proof}

\section{Examples: $C^*$-algebras and Schatten spaces}\label{ss:examples} 
We provide a few examples of operator spaces which are regular, that is, normal and generating.

\begin{prop}\label{op sys}
Any operator system is $1$-matricial normal and $1$-matricial generating.
\end{prop}

\begin{proof}[Sketch of a proof]
Any operator system $\cc X$ can be viewed as embedded into $B(H)$. Normality is is inherited from $B(H)$. To show that $\cc X$ is generating, recall that, for $x \in S_\infty^n(\cc X)$,
$$
\|x\| = \inf \Big\{ \lambda : \begin{pmatrix}  \lambda I_n \otimes I_H & x  \\  x^* & \lambda I_n \otimes I_H  \end{pmatrix} \geq 0 \Big\} ,
$$
where $I_H$ stands for the identity in $B(H)$.
\end{proof}

\begin{prop}\label{C*-alg}
Any $C^*$-algebra is $1$-matricial normal and $1$-matricial generating.
\end{prop}

\begin{proof}[Sketch of a proof]
Again, normality is straightforward. To establish generation, note that, for $a \in S_\infty^n(\cc A)$ (here, $\cc A$ is a $C^*$-algebra), we have $\begin{pmatrix}  (aa^*)^{1/2} & a  \\  a^* & (a^* a)^{1/2} \end{pmatrix} \geq 0$.
\end{proof}


Next we turn our attention to non-commutative sequence spaces. Fix $p \in [1,\infty)$, and find $q$ so that $1/p + 1/q = 1$. Consider the Schatten space $\cc X = S_p$, with its standard operator space structure \cite{Pisier-Lp}. Equip $\cc X$ with ``parallel duality'' -- that is, for $x \in \cc X$ and $x^\flat \in \cc X^\flat$ (on the Banach space level, $\cc X^\flat$ coincides with $S_q$), define the dual action via $\langle x^\flat , x \rangle = \tra( x^\flat x^T)$, where $x^T$ stands for the transpose of $x$. 

To define positive cones, write a generic $x \in S_\infty^n(\cc X)$ as $ \sum_{i,j=1}^n E_{ij} \otimes x_{ij} \in S_\infty^n(\cc X)$, with $x_{ij} = \sum_{k,\ell=1}^m [x_{ij}]_{k\ell} E_{k\ell}$). We say that $x \in S_\infty^n(\cc X)^+$ if the matrix $\sum [x_{ij}]_{k\ell} E_{ij} \otimes E_{k\ell}$, where the rows and the columns are indexed by pairs $(i,k)$ and $(j,\ell)$, respectively, is positive.
Clearly our requirements are satisfied: (i) if $x \geq 0$ and $-x \geq 0$, then $x=0$; (ii) if $x \geq 0$, then $a \cdot x \cdot a^* \geq 0$ for any matrix $a$, and (iii) if $x \in S_\infty^m(S_p)^+, y \in S_\infty^n(S_p)^+$, then $x \oplus y \in S_\infty^{m+n}(S_p)^+$.

\begin{prop}\label{Sp cone}
In the above notation, $\cc X^\flat = S_q$,  $1/p + 1/q = 1$, with the matricial order on $S_q$ defined as for $S_p$.
\end{prop}

\begin{proof}
It is well known that such identification holds for operator space structure. To handle the matricial order, consider $x^\flat = (x^\flat_{ij})_{i,j=1}^n = \sum_{i,j=1}^n E_{ij} \otimes x^\flat_{ij} \in S_\infty^n(\cc X^\flat)$. As we did with $x$, write $x^\flat_{ij} = \sum_{k\ell} [x_{ij}^\flat]_{k\ell} E_{k\ell}$, then $x = \sum_{i,j=1}^n \sum_{k\ell} [x_{ij}^\flat]_{k\ell} E_{ij} \otimes E_{k\ell}$ -- that is, we view $x$ as a matrix with entries $\big( [x_{ij}^\flat]_{k\ell} \big)$, with pairs $(i,k)$ and $(j,\ell)$ indexing the rows and columns, respectively.
   
   We shall show that $x^\flat \in S_\infty^n(\cc X^\flat)^+$ iff the corresponding infinite matrix $\big( [x_{ij}^\flat]_{k\ell} \big)$ is positive.
   Indeed, $x^\flat \in S_\infty^n(\cc X^\flat)^+$ iff the operator $T : \cc X \to S_\infty^n : x \mapsto \sum_{i,j=1}^n \langle x_{ij}^\flat , x \rangle E_{ij}$ (where $(E_{ij})$ stand for the matrix units), is completely positive.
   This, in turn, is equivalent to the corresponding ``formal'' Choi matrix $C_T = \sum_{k,\ell=1}^\infty T(E_{k\ell}) \otimes E_{k\ell}$ being positive (see e.g.~\cite[Section 2.3]{AliceBobBanach}). Note that $T(E_{k\ell}) = \sum_{i,j=1}^n [x_{ij}^\flat]_{k\ell} E_{ij}$ (where $[x_{ij}^\flat]_{k\ell}$ denotes the $(k,\ell)$ entry of $x_{ij}^\flat$), hence
   $$ C_T = \sum_{k,\ell=1}^\infty \sum_{i,j=1}^n [x_{ij}^\flat]_{k\ell} E_{ij} \otimes E_{k\ell} = U^* \cdot \Big( \sum_{k,\ell=1}^\infty \sum_{i,j=1}^n [x_{ij}^\flat]_{k\ell} E_{k\ell} \otimes E_{ij} \Big) \cdot U = U^* \cdot x^\flat \cdot U , $$
where $U$ is the ``flip'' map -- that is, $U | j \ell \rangle = |\ell j \rangle$. Clearly $U^* \cdot x^\flat \cdot U$ is positive iff $x^\flat$ is.
\end{proof}

\begin{prop}\label{schatten}
    For $1 \leq p \leq \infty$, the Schatten space $S_p^m$ is $1$-matricial normal and $1$-matricial generating.
\end{prop}

\begin{proof}
Show first that $S_p$ is matricial normal. Suppose, for the sake of contradiction, there exist $x, x_1, x_2 \in S_\infty^n(S_p)$ so that $\|x\| > 1 > \|x_1\| \vee \|x_2\|$, and $\begin{pmatrix} x_1 & x \\ x^* & x_1 \end{pmatrix} \in S_\infty^{2n}(S_p)^+$.
Find $a,b \in S_{2p}^n$ so that $\|a\|_{2p}, \|b\|_{2p} < 1$, and $\|a \cdot x \cdot b\|_p > 1$ (here $a \cdot x \cdot b$ is viewed as an element of $S_p^n[S_p] = S_p(\ell_2^n \otimes_2 \ell_2)$). Now note that
$$
\begin{pmatrix} a x_1 a^* & axb \\ b^*x^*a^* & b^* x_2 b \end{pmatrix} = \begin{pmatrix} a^* & 0 \\ 0 & b \end{pmatrix}^* \begin{pmatrix} x_1 & x \\ x^* & x_1 \end{pmatrix} \begin{pmatrix} a^* & 0 \\ 0 & b \end{pmatrix} \geq 0 .
$$
By \cite{horn-mathias}, $\|axb\|_p^2 \leq \|ax_1a^*\|_p \|b^*x_2b\|_p$. However, $\|axb\|_p > 1$ by our choice of $a$ and $b$, while $\|ax_1a^*\|_p \leq \|a\|_{2p}^2 \|x_1\| < 1$, and similarly, $\|b^* x_2 b\|_p < 1$. This is the desired contradiction.

Duality (explained in \Cref{Sp cone}) shows that $S_p = S_q^\flat$ for $1 \leq p < \infty$ and $1/p + 1/q = 1$. Consequently, for finite $p$, $S_p$ is generating, as a dual of a normal space. The case of $p=\infty$ has been handled in \Cref{C*-alg}.
\end{proof}

\begin{rmk}
    (1) \cite{Schreiner-DISS} gives another proof of the $1$-normality and $1$-generation of the Schatten space . The same also holds for the spaces $L_p(\mu)$.

    (2) \cite[Section 5.2]{Schreiner-DISS} shows that, for a self-adjoint $x \in S_\infty^n(\cc X)$:

    (i)     $\displaystyle \|x\|_{S_p^n[\cc X]} = \inf \big\{ \|a\|_{2p}^2 \|u\|_{S_\infty^n(\cc X)} : a \in S_{2p}^n, u \in S_\infty^n(\cc X)_h, x = a^* \cdot u \cdot a \big\}$.

    (ii)     $\displaystyle \|x\|_{S_\infty^n[\cc X]} = \sup \big\{ \|a^* \cdot x \cdot a\|_{S_p^n[\cc X]} : a \in S_{2p}^n, \|x\|_{2p} \leq 1 \big\}$.
\end{rmk}

\begin{rmk}
It is easy to see that, for $n \geq 2$, $S_\infty^n(S_p)^+$ (here $S_p$ is equipped with its ``natural'' matricial order structure, described above) does not coincide with $S_\infty^n(\MIN(S_p))^+$. Indeed, take $x = \sum_{i,j=1}^2 E_{ij} \otimes E_{ji} \in S_\infty^n(S_p)$.
This $x$ is represented by a $4 \times 4$ matrix which is not positive, hence $x \notin S_\infty^n(S_p)^+$. On the other hand, for any $\xi = (\xi_i) \in \ell_2^n$, $\langle \xi |x| \xi \rangle = \begin{pmatrix} |\xi_1|^2 & \overline{\xi_1} \xi_2 \\ \xi_1 \overline{\xi_2} & |\xi_2|^2 \end{pmatrix} \geq 0$, so $x \in S_\infty^n(\MIN(S_p))^+$.
By duality, $S_\infty^n(S_p)^+$ does not coincide with $S_\infty^n(\MAX(S_p))^+$. 

Similarly, if $A$ is a non-abelian $C^*$-algebra, then, for $n \geq 2$,
\begin{equation}
\label{eq:inclusions}
S_\infty^n(\MAX(A))^+ \subsetneq S_\infty^n(A)^+ \subsetneq S_\infty^n(\MIN(A))^+ 
\end{equation}
(the cone in the middle refers to the canonical matricial order structure of $A$). Indeed, by passing to the bidual, we can assume that $A$ is a von Neumann algebra; 
then it must contain $S_\infty^2$ as a subalgebra, which is the range of a completely positive, completely contractive projection.
The reasoning we applied for $S_p$ can be ``recycled'' to establish \eqref{eq:inclusions}.

The situation is different if $A$ is an abelian $C^*$-algebra, that it, a $C(K)$ space (a Banach lattice). By \cite[Theorem 3.1]{Wittstock}, the maximal and the minimal matricial cones coincide.
\end{rmk}

\section{Matricial order structures on Banach lattices}\label{lattice}

Here we describe how to equip a Banach lattice with a matricial operator space structure, and examine some properties of the said structure. For Banach lattice background see, for instance, \cite{A-A}, \cite{Al-Bur}, or \cite{M-N}.

Following \cite[Section 3.2]{A-A} (where more detail can be found), we say that the \emph{complexification} $\cc X$ of a real Banach lattice $\cc X_\R$ consists of formal linear combinations $x + \iota y$ ($x, y \in \cc X_\R$), with $|x + \iota y|:= \vee_\theta (\cos \theta \, x + \sin \theta \, y)$, and $\|x + \iota y\| = \| |x + \iota y| \|$. The $*$-operation (involution) takes $x + \iota y$ to $x - \iota y$. 

We should note that, for Banach spaces, a number of ``reasonable'' complexification procedures exist, see \cite{complexifications}.

If $\cc X$ and $\cc Y$ are complex Banach lattices, we say that a map $T$ is a \emph{lattice homomorphism} if it takes $\cc X_\R$ to $\cc Y_\R$, and preserves lattice operations. $T$ is a \emph{lattice isomorphism} if it is bijective, and both $T$ and $T^{-1}$ are lattice homomorphisms.

Below, we discuss the minimal and maximal matricial order structures on Banach lattices. Although we do not use this, let's nevertheless recall a result of \cite[Section 3]{Wittstock}: $S_\infty^n(\MIN(\cc X))^+$ coincides with the closure of finite sums $\sum_{k=1}^N a_k \otimes x_k$, with $a_k \in S_\infty^{n+}, x_k \in \cc X^+$. Consequently, the ``minimal'' and ``maximal'' matricial cones over a Banach lattice coincide.

\begin{lem}\label{lattice MIN nice}
Any Banach lattice $\cc X$ is $\MIN$-nice.
\end{lem}

\begin{proof}
We only need to verify that \eqref{eq:min-nice} implies $\|x\| \leq \|x_1\| \vee \|x_2\|$. Krivine functional calculus (see e.g.~\cite[Section 1.d]{LT2}) shows that $\sup_{|\omega| = 1} \sup \Re(\omega x) = |x|$, hence \eqref{eq:min-nice} implies that, for any $t, s \in \R$, $t^2 x_1 + s^2 x_2 \geq 2 ts |x|$. Restricting to $t = 1/s$, we obtain
$$    |x| \leq \frac12 \inf_{t \in \R} \Big( t^2 x_1 + \frac1{t^2} x_2 \Big) = \sqrt{x_1 x_2} ,    $$
where the right hand side is again defined using Krivine functional calculus. By \cite[Proposition 1.d.2]{LT2}, $\|x\| \leq \| \sqrt{x_1 x_2}\| \leq \sqrt{\|x_1\| \|x_2\|}$.
\end{proof}

\begin{lem}\label{lattice MAX nice}
Any Banach lattice $\cc X$ is $\MAX$-nice.
\end{lem}

It may be possible to deduce this result from \Cref{lattice MIN nice} by duality; however, we present an independent proof.

\begin{proof}
Consider $x \in \cc X$, with $\|x\| < 1$. Denote by $I_x$ the principal ideal generated by $x$ -- that is, the linear (but not necessarily closed) subspace of $\cc X$ consisting of elements $y$ for which there exists $c \geq 0$ with $|y| \leq c |x|$.
By \cite[Corollary 3.21]{A-A}, there exists a Hausdorff compact $K$ and a bijective contractive lattice homomorphism $T : (C(K), \| \cdot \|_\infty) \to (I_x, \| \cdot \|)$, taking $\bf 1$ (the identity on $K$) to $|x|$; so, $|T^{-1} x| = {\bf 1}$.
A ``partition of identity'' argument allows us to find non-negative $g_1, \ldots, g_N \in C(K)$ and $t_1, \ldots, t_N \in K$ so that $\sum_k g_k = \bf 1$, and
$$
\max_{t \in K} \big| [T^{-1}x](t) - \sum_k [T^{-1}x](t_k) g_k(t) \big| \leq \vr .
$$
Let $\alpha_k = [T^{-1}x](t_k)$, $\omega_k = \alpha_k/|\alpha_k|$, $\eta_k = \sqrt{|\alpha_k|}$, $\xi_k = \omega_k \sqrt{|\alpha_k|}$ (so $\alpha_k = \xi_k \overline{\eta_k}$), and $x_k = T g_k$.
Then $\sum_k |\xi_k|^2 x_k = \sum_k |\eta_k|^2 x_k = |x|$. Applying $T$ to $\big| T^{-1}x - \sum_k \alpha_k g_k \big| \leq \vr \bf 1$, we obtain $\big| x - \sum_k \xi_k \overline{\eta_k} x_k \big| \leq \vr |x|$, hence $\big\| x - \sum_k \xi_k \overline{\eta_k} x_k \big\| \leq \vr \|x\| < \vr$.
\end{proof}

Combining the above lemmas with results of \Cref{min max structures}, conclude:

\begin{cor}
For any Banach lattice $\cc X$, $\MIN(\cc X)$ is $1$-matricial normal, and $\MAX(\cc X)$ is $1$-matricial generating.
\end{cor}

\begin{prop}\label{AM-generating}
If $\cc X$ is a Banach lattice, then $\MIN(\cc X)$ is matricial generating if only if it is lattice isomorphic to an AM-space.
\end{prop}

Recall that a \emph{real AM-space} is a sublattice of a real-valued $C(K)$ space; a \emph{complex AM-space} is a complexification of such. By \cite[Theorem 2.1.12]{M-N}, a Banach lattice $\cc X$ is lattice isomorphic to an AM-space iff there exists a constant $C$ so that the inequality $\|\sum_{i=1}^n x_i\| \leq C \max_i \|x_i\|$ holds whenever $(x_i)_{i=1}^n \subset \cc X$ are disjoint. 


For future use, we define a dual object: a \emph{real AL-space} is a Banach lattice $\cc X$ where $\|x_1 + x_2\| = \|x_1\| + \|x_2\|$ whenever $x_1, x_2 \in \cc X$ are disjoint; a \emph{complex AL-space} is a complexification of a real one.
Theorem 1.b.2 of \cite{LT2} tells us that any real AL-space can be identified with $L_1(\mu)$, for some measure $\mu$. In light of \cite[Section 3.2, Exercise 5]{A-A}, the same holds in the complex setting. 
Further, by \cite[Theorem 1.b.12]{LT2}, $\cc X$ is lattice isomorphic to an AL-space iff there exists a constant $c>0$ so that, for any pairwise disjoint $x_1, \ldots, x_n \in \cc X$, we have $\|\sum_i x_i\| \geq c \sum_i \|x_i\|$. 
By e.g.~\cite[Theorem 4.23]{Al-Bur}, $\cc X$ is a real AL-space iff its dual is a real AM-space, and vice versa. By \cite[Corollary 3.26]{A-A}, $(\cc X_\R + \iota \cc X_\R)^\flat$ is naturally identified with the complexification of $\cc X_\R^\flat$, hence the duality between AL and AM spaces remains true in the complex setting.

\begin{proof}[Proof of \Cref{AM-generating}, backward direction]
\cite[Theorems 1.b.5-6]{LT2} show that any real AM-space has a canonical representation as a sublattice of $C(K)$, for some Hausdorff compact $K$: there exists a collection of triples $(t_i, s_i, \lambda_i)_{i \in I}$, with $\lambda_i \geq 0$ and $t_i, s_i \in K$ so that $\cc X$ consists of all $x \in C(K)$ which satisfy $x(t_i) = \lambda_i x(s_i)$ for any $i \in I$. A combination of Lemma 3.18 and Exercise 3.2.3 from \cite{A-A} extends this result to the complex setting.

For $a \in S_\infty^n(\cc X)_h$, define $b : K \to S_\infty^n : t \mapsto (a^*(t) a(t))^{1/2}$.
Clearly $b(t_i) = \lambda_i b(s_i)$ for any $i \in I$, so $b \in S_\infty^n(\cc X)$. 
It is easy to see that $\|b\| = \sup_{t \in K} b(t) = \sup_{t \in K} |a(t)| = \|a\|$. Finally, $\begin{pmatrix} b(t) & a(t) \\ \overline{a(t)} & b(t) \end{pmatrix} \in S_\infty^{2n+}$ for any $t \in K$, hence $\begin{pmatrix} b & a \\ a^* & b \end{pmatrix} \in S_\infty^{2n}(\cc X)^+$. From this, the $1$-matricial generation follows.
\end{proof}

For the forward direction we need a lemma.

\begin{lem}\label{expectation}
Suppose $p$ is an orthogonal projection onto a subspace of $\ell_2^n$, and $p^\perp$ is its orthogonal complement. If $a \in S_\infty^n$ is such that $a \geq \pm(p - p^\perp)$, then $\expe \langle \xi | a | \xi \rangle \geq 1$; here, $\expe$ is the expected value computed when $\xi$ is uniformly distributed over the unit sphere of $\ell_2^n$.
\end{lem}

\begin{proof}
Clearly $\langle \xi | a | \xi \rangle = \langle p\xi | a | p\xi \rangle + \langle p^\perp\xi | a | p^\perp\xi \rangle  + \langle p\xi | a | p^\perp\xi \rangle  + \langle p^\perp\xi | a | p\xi \rangle$.
Our probability measure on the sphere is invariant under the orthogonal transformation $\xi \mapsto p\xi - p^\perp \xi$, hence $\expe \langle p\xi | a | p^\perp\xi \rangle = \expe \langle p\xi | a | - p^\perp\xi \rangle$, and consequently, $\expe \langle p\xi | a | p^\perp\xi \rangle = 0$. Likewise, $\expe \langle p^\perp\xi | a | p\xi \rangle = 0$.

For any $\xi \in \ell_2^n$, we have
$$
\langle p\xi | a | p\xi \rangle \geq \langle p\xi | (p - p^\perp) | p\xi \rangle = \|p\xi\|^2 .
$$
Similarly, $\langle p^\perp\xi | a | p^\perp\xi \rangle \geq \|p^\perp \xi\|^2$. Therefore,
\begin{align*}
\expe \langle \xi | a | \xi \rangle 
&
= \expe \langle p\xi | a | p\xi \rangle + \expe \langle p^\perp\xi | a | p^\perp\xi \rangle  + \expe \langle p\xi | a | p^\perp\xi \rangle  + \expe \langle p^\perp\xi | a | p\xi \rangle
\\ &
= \expe \big( \langle p\xi | a | p\xi \rangle + \langle p^\perp\xi | a | p^\perp\xi \rangle \big) \geq \expe \big( \|p \xi\|^2 + \|p^\perp \xi\|^2 \big) = 1 . \qedhere
\end{align*}
\end{proof}

\begin{proof}[Proof of \Cref{AM-generating}, forward direction]
If $\cc X$ is not an AM-space, then, by \cite[Lemma 4]{CL75}, for any $K > 0$ we can find disjoint $x_1, \ldots, x_n \in \cc X^+$ so that $\|\sum_k x_k\| > K$, while $\sum_k | \langle x^\flat , x_k \rangle |^2 \leq 1$ for any $x^\flat \in \ball(\cc X^\flat)$.
The second condition means that the operator $\ell_2^n \to \cc X : \delta_i \mapsto x_i$ ($(\delta_i)_{i=1}^n$ is the canonical basis of $\ell_2^n$) is contractive.

By \cite[Section 11.1]{AliceBobBanach}, for $N$ large enough $S_\infty^N$ contains a real subspace of dimension $n$ consisting exclusively of real scalar multiples of symmetric orthogonal matrices.
Let $U_1, \ldots, U_n$ be a normalized basis in this space, so that ${\textrm{tr}}(U_iU_j) = 0$ for $i \neq j$. Then $\|\sum_{i=1}^n U_i \otimes x_i\| \leq 2$, since $\sum_{i=1}^n U_i \otimes x_i$ can be viewed as an operator from $\Span[U_i : 1 \leq i \leq n]$ to $\cc X$, and the former space is $2$-isomorphic to $\ell_2^n$. We shall show that, if $S_\infty^N(\cc X) \ni a \geq \pm \sum_{i=1}^n U_i \otimes x_i$, then 
\begin{equation}
\expe \langle \xi | a | \xi \rangle \geq \sum_i x_i ,
\label{eq:desired}    
\end{equation}
with the expected value taken for $\xi$ uniformly distributed over the unit sphere of $\ell_2^N$. For any such $\xi$, $\langle \xi | a | \xi \rangle \leq \|a\|$, hence $\|a\| \geq \|\sum_i x_i\| \geq K$, showing that $\MIN(\cc X)$ cannot be matricial generating.

Write $a = \sum_{i,j=1}^N E_{k\ell} \otimes y_{k\ell}$. As in the proof of \Cref{lattice MAX nice}, we can consider $x_i, y_{k\ell}$ as ``living'' in $C(K)$, corresponding to the ideal $I_x$ with $x = \sum_i x_i + \sum_{k,\ell} |y_{k\ell}|$.
In particular, we view $x_i$'s as disjoint positive functions on $K$, and $a$ is a matrix-valued function on $K$. For any $t \in {\textrm{supp}} x_i$, $a(t) \geq \pm x_i(t) U_i$. Write $U_i = p_i - p_i^\perp$, where $p_i$ is an orthogonal projection. By \Cref{expectation}, $\expe \langle \xi | a(t) | \xi \rangle \geq x_i(t)$ for $t \in {\textrm{supp}} x_i$. Summing over $i$, we obtain \eqref{eq:desired}.
\end{proof}

By duality, conclude:

\begin{cor}
Suppose $\cc X$ is a Banach lattice. Then $\MAX(\cc X)$ is matricial normal if and only if $\cc X$ is lattice isomorphic to an AL-space.
\end{cor}

\printbibliography
\end{document}